\documentclass[a4paper]{amsart}

\usepackage{version}
\includeversion{journal}
\excludeversion{preprint}

\usepackage{fullpage}
\usepackage{amssymb,amsmath}
\usepackage{pst-all}
\usepackage[neveradjust]{paralist}
\usepackage[all]{xypic}
\xyoption{dvips}
\usepackage{url}
\usepackage[colorlinks,plainpages,backref]{hyperref}
\usepackage{dsfont}
\renewcommand{\mathbb}{\mathds}
\usepackage{mathrsfs}
\DeclareMathAlphabet{\mathsc}{U}{rsfs}{m}{n}
\renewcommand{\mathcal}{\mathsc}

\theoremstyle{definition}
\newtheorem{ntn}{Notation}[section]
\newtheorem{dfn}[ntn]{Definition}
\newtheorem{add}[ntn]{Addendum}
\theoremstyle{plain}
\newtheorem{lem}[ntn]{Lemma}
\newtheorem{prp}[ntn]{Proposition}
\newtheorem{thm}[ntn]{Theorem}
\newtheorem{mth}[ntn]{Theorem}

\newtheorem{cor}[ntn]{Corollary}
\newtheorem{cnj}[ntn]{Conjecture}

\theoremstyle{remark}

\newtheorem{rmk}[ntn]{Remark}
\newtheorem{exa}[ntn]{Example}

\numberwithin{equation}{section}

\newcommand{\A}{\mathcal{A}}
\newcommand{\CC}{\mathbb{C}}
\newcommand{\KK}{\mathbb{K}}
\newcommand{\LL}{\mathbb{L}}
\newcommand{\NN}{\mathbb{N}}
\newcommand{\OO}{\mathcal{O}}

\newcommand{\RR}{\mathbb{R}}

\newcommand{\mm}{\mathfrak{m}}

\newcommand{\ideal}[1]{{\left\langle#1\right\rangle}}
\newcommand{\into}{\hookrightarrow}

\newcommand{\onto}{\twoheadrightarrow}
\newcommand{\p}{\partial}
\newcommand{\rdots}{\mathinner{%
  \mkern1mu\raise1pt\hbox{.}%
  \mkern2mu\raise4pt\hbox{.}%
  \mkern2mu\raise7pt\vbox{\kern7pt\hbox{.}}\mkern1mu}}
\renewcommand{\red}{\mathrm{red}}
\newcommand{\rel}{\mathrm{rel}}

\newcommand{\xymat}{\SelectTips{cm}{}\xymatrix}

\newcommand{\beq}{\begin{equation}}
\newcommand{\eeq}{\end{equation}}
\newcommand{\ben}{\begin{enumerate}}
\newcommand{\een}{\end{enumerate}}

\DeclareMathOperator{\ad}{ad}

\DeclareMathOperator{\codim}{codim}
\DeclareMathOperator{\coker}{coker}
\DeclareMathOperator{\diag}{diag}
\DeclareMathOperator{\Der}{Der}
\DeclareMathOperator{\End}{End}
\DeclareMathOperator{\Fit}{F}
\DeclareMathOperator{\GL}{GL}
\DeclareMathOperator{\grade}{grade}
\DeclareMathOperator{\Hess}{Hess}
\DeclareMathOperator{\Hom}{Hom}
\DeclareMathOperator{\id}{id}
\DeclareMathOperator{\img}{im}
\DeclareMathOperator{\Lie}{Lie}

\DeclareMathOperator{\OG}{O}
\DeclareMathOperator{\Sing}{Sing}
\DeclareMathOperator{\Spec}{Spec}

\begin{document}

\title[Coxeter arrangements and discriminants]
{Partial normalizations of\\Coxeter arrangements and discriminants}

\author{Michel Granger}
\address{
Michel Granger\\
Universit\'e d'Angers, D\'epartement de Math\'ematiques\\
LAREMA, CNRS UMR n\textsuperscript{o}6093\\
2 Bd Lavoisier\\
49045 Angers\\
France
}
\email{granger@univ-angers.fr}
\thanks{}

\author{David Mond}
\address{
D. Mond\\
Mathematics Institute\\
University of Warwick\\
Coventry CV47AL\\
England
}
\email{D.M.Q.Mond@warwick.ac.uk}
\thanks{}

\author{Mathias Schulze}
\address{
M. Schulze\\
Department of Mathematics\\
Oklahoma State University\\
Stillwater, OK 74078\\
United States}
\email{mschulze@math.okstate.edu}
\thanks{}

\date{\today}


\dedicatory{Amended Postprint}

\begin{abstract}
We study natural partial normalization spaces of Coxeter arrangements and discriminants and relate their geometry to representation theory.
The underlying ring structures arise from Dubrovin's Frobenius manifold structure which is lifted (without unit) to the space of the arrangement.
We also describe an independent approach to these structures via duality of maximal Cohen--Macaulay fractional ideals.
In the process, we find 3rd order differential relations for the basic invariants of the Coxeter group.
Finally, we show that our partial normalizations give rise to new free divisors.

\medskip\noindent
Addendum~\ref{120} contains a proof due to H.~Terao of Conjecture~\ref{118} 
and Addendum~\ref{121} points out a simpler proof of Theorem~\ref{68}.
\end{abstract}

\maketitle
\tableofcontents

\section*{Introduction}

V.I.~Arnol'd was the first to identify the singularities of type $ADE$, that is $A_\ell$, $D_\ell$, $E_6$, $E_7$ or $E_8$, as the simple singularities -- those that are adjacent to only finitely many other types. 
He also uncovered the links between the Coxeter groups of type $B_\ell$, $C_\ell$ and $F_4$ and boundary singularities, see \cite{Arn79}. 
His formul\ae\ for generators of the module of logarithmic vector fields $\Der(-\log D)$ along the discriminant $D$ parallels K.~Saito's definition of free divisors.
Along with Brieskorn, Dynkin, Gelfan'd, and Gabriel, Arnol'd revealed the $ADE$ list as one of the central {piazzas} in mathematical heaven, where representation theory, algebra, geometry and topology converge. 
As with so many of Arnol'd's contributions, his work on this topic has given rise to a huge range of further work by others.

Let $f\colon X=(\CC^n,0)\to(\CC,0)=S$ be a complex function singularity of type $ADE$ and let $F\colon X\times B\to S$ be a miniversal deformation of $f$ with base $B=(\CC^\mu,0)$.
Writing $f_u:=F(-,u)$, the discriminant $D\subset B$ is the set of parameter values $u\in B$ such that $f_u^{-1}(0)$ is singular.
It is isomorphic to the discriminant of the Coxeter group $W$ of the same name.
Here the discriminant is the set of exceptional orbits in the orbit space $V/W$.
This is only the most superficial feature of the profound link between singularity theory and the geometry of Coxeter groups which Arnol'd helped to make clear. 

The starting point of this paper is the fact, common to Coxeter groups and singularities, that $D$ is a free divisor (see e.g.~\cite[\S 4.3]{Her02}) with a symmetric Saito matrix $K$ whose cokernel is a ring in the singularity case.
By definition the Saito matrix $K$ is the $\mu\times\mu$-matrix whose columns are the coefficient vectors of a basis of $\Der(-\log D)$ with respect to a basis of the module $\Der_B:=\Der_\CC(\OO_B)$ of vector fields on $B$.

On the singularity theory side these two roles are well known.
Let $h$ be a defining equation for $D$ and $J_D$ the Jacobian ideal of $D$. 
Then $K$ appears in the exact sequence 
\[
\xymat{0\ar[r]&\OO_B^\mu\ar[r]^-K& \Der_{B}\ar[r]^-{dh}&J_D\ar[r]&0}
\]
which defines $\Der(-\log D)$ as the vector fields which preserve the ideal of $D$. 

Let $\Sigma\subset X\times B$ be the relative critical locus, defined by the Jacobian ideal $J^\rel_F$ of $F$ relative to $B$, and let $\Sigma^0:=\Sigma\cap V(F)$. 
Let $\pi:\Sigma\to B$ denote the restriction of the projection $X\times B\to B$,  so that $D=\pi(\Sigma^0)$. 
Then $K$ also appears in the exact sequence
\beq\label{113}
\xymat{0\ar[r]&\OO_B^\mu\ar[r]^-K& \Der_{B}\ar[r]^-{dF}&\pi_*\OO_{\Sigma^0}\ar[r]&0}
\eeq
in which $dF$ maps a vector field $\eta\in\Der_B$ to the function $dF(\tilde{\eta})$ on $\Sigma^0$, where $\tilde\eta$ is a lift of $\eta$ to $X\times B$. 
As $\pi_*\OO_\Sigma$ is free over $\OO_B$ of rank $\mu$, we can make the identifications 
\[
\pi_*\OO_\Sigma\cong\OO_B^\mu\cong\Der_B,
\]
and reinterpret $K$ as the matrix of the $\OO_B$-linear operator induced  on $\pi_*\OO_\Sigma$ by multiplication by $F$, whose cokernel is also, evidently, $\pi_*\OO_{\Sigma^0}$.

Similar to the case of $ADE$ singularities and corresponding Coxeter groups, Coxeter groups of type $B_k$ and $F_4$ are linked with boundary singularities, for which a similar argument shows that the cokernel of $K$ is naturally a ring.
Also for these and the remaining Coxeter groups $I_2(k)$, $H_3$ and $H_4$, the cokernel of $K$ carries a natural ring structure. 
The simplest way to see this involves the Frobenius structure constructed on the orbit space by Dubrovin in \cite{Dub98}, following K.~Saito. 
Here the key ingredient is a fiber-wise multiplication on the tangent bundle, which coincides with the multiplication coming from $\OO_\Sigma$ in the $ADE$ singularity case. 
We recall the necessary details of Dubrovin's construction, following C.~Hertling's account in \cite{Her02}, in Section~\ref{190}, in preparation for the proof of our main result.
This states that also the cokernel of a transposed Saito matrix for the reflection arrangement of a Coxeter group carries a natural ring structure.

\begin{mth}\label{106}\
\ben

\item\label{106a} Let $\A$ be the reflection arrangement of a Coxeter group $W$ acting on the vector space $V\cong\CC^\ell$, let $p_1,\dots, p_\ell$ be generators of the ring of $W$-invariant polynomials, homogeneous in each irreducible component of $V$, and let $J$ be the Jacobian matrix of the map $(p_1,\dots, p_\ell)$ (which is in fact a transposed Saito matrix for $\A$).
Then $\coker J$ has a natural structure of $\CC[V]$-algebra.

\item\label{106b} Denoting $\Spec\coker J$ by $\tilde\A$, we have
\ben
\item\label{106bi} $\tilde\A$ is finite and birational over $\A$ (and thus lies between $\A$ and its normalization). 
\item\label{106bii} For $x\in\A$, let $W_x$ be the stabilizer of $x$ in $W$, let $X(x)$ be the flat of $\A$ containing $x$, and let $\{\A_{x,i}\mid i\in I_x\}$ be the set of reflection arrangements of the irreducible summands in the representation of $W_x$ on $V/X(x)$.
Then, locally along $X(x)$, $\tilde\A$ can be identified with the disjoint union $\coprod_{i\in I_x}\tilde\A_{x,i}\times X(x)$ of connected spaces.
In particular, the geometric fiber of $\tilde\A\to\A$ over $x$ is indexed by $I_x$. 
\item\label{106biii} Under the bijection of \eqref{106bii}, smooth points of $\tilde\A$ correspond to representations of type $A_1$.  
\een

\een
\end{mth}

\begin{exa}\label{108}\
\begin{asparaenum}

\item\label{108a} In the case of $A_2$, the arrangement $\A$ consists of three concurrent coplanar lines. 
In this case  $\tilde\A$ is isomorphic to the union $L_2$ of the three coordinate axes in $3$-space. 
One can check this rather easily: $L_2$ is the only connected curve singularity mapping finitely and birationally to $\A$, but which is not isomorphic to it.  
More generally, in the case of $A_\ell$, with $\begin{pmatrix}\ell+1\\2\end{pmatrix}$ reflecting  hyperplanes, $\tilde\A$ is isomorphic to the codimension-$2$ subspace arrangement $L_\ell$ in $(\ell+1)$-space consisting of the $(\ell-1)$-planes $L_{i,j}:=\{x_i=x_j=0\}$ for $1\leq i<j\leq \ell+1$. 
The projection $x\mapsto x-x^\sharp$, where $x^\sharp$ is $x$ averaged by the action of the symmetric group $S_\ell$ permuting coordinates, gives an $S_\ell$-equivariant map of $L_{\ell}$ to the standard arrangement $\A\subset\{\sum_{i=1}^{\ell+1}x_i=0\}$, sending $L_{i,j}$ isomorphically to $\{x_i=x_j\}$. 
We return to this example, and prove these assertions, in Subsection~\ref{109}. 

\item\label{108b} Figure~\ref{107} shows a $2$-dimensional section of the hyperplane arrangement $\A$ for $A_3$, on the left, and, on the right, a topologically accurate view of the preimage of this section in $\tilde \A$.

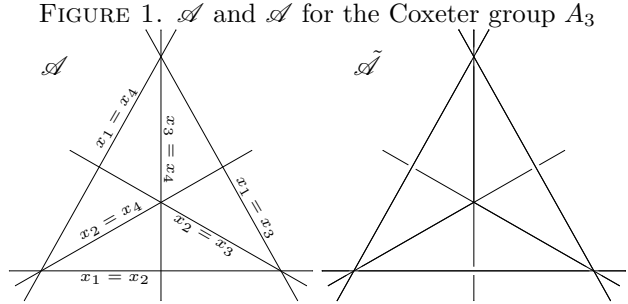
\begin{figure}[ht]
\caption{$\A$ and $\tilde\A$ for the Coxeter group $A_3$}\label{107}
\begin{center}
\begin{pspicture}(4,3.6)
\psset{linewidth=0.1mm,arrows=cc-cc}
\psline(2,0)(2,3.6)
\psline(0.08,0.2)(3.2,2)
\psline(3.92,0.2)(0.8,2)
\psline(0,0.4)(4,0.4)
\psline(0.2,0)(2.2,3.6)
\psline(3.8,0)(1.8,3.6)
\rput[t]{-30}(2.6,0.92){\tiny$x_2=x_3$}
\rput[b]{30}(1.38,1){\tiny$x_2=x_4$}
\rput[b]{-90}(2.04,2){\tiny$x_3=x_4$}
\rput[t]{0}(1.4,0.36){\tiny$x_1=x_2$}
\rput[b]{60}(1.48,2.4){\tiny$x_1=x_4$}
\rput[b]{-60}(3.18,1.2){\tiny$x_1=x_3$}
\rput[b]{0}(0.6,3){$\A$}
\end{pspicture}
\begin{pspicture}(4,3.6)
\psset{linewidth=0.1mm,arrows=cc-cc}
\psline(2,0)(2,3.6)
\psline(0.08,0.2)(3.2,2)
\psline(3.92,0.2)(0.8,2)
\psline[border=0.5mm](0,0.4)(4,0.4)
\psline[border=0.5mm](0.2,0)(2.2,3.6)
\psline[border=0.5mm](3.8,0)(1.8,3.6)
\psline(0,0.4)(4,0.4)
\psline(0.2,0)(2.2,3.6)
\psline(3.8,0)(1.8,3.6)
\psline(2,1.31)(2,3.6)
\psline(2,1.31)(0.08,0.2)
\psline(2,1.31)(3.92,0.2)
\rput[b]{0}(0.6,3){$\tilde\A$}
\end{pspicture}
\end{center}
\end{figure}

The planes $\{x_{i_1}=x_{i_2}\}$ and $\{x_{i_3}=x_{i_4}\}$ meet orthogonally if $i_1$, $i_2$, $i_3$ and $i_4$ are all different, and the reflections in these planes commute; it follows that at a point $x$ in the stratum $\{x_{i_1}=x_{i_2}\neq x_{i_3}=x_{i_4}\}$, the representation is of type $A_1\oplus A_1$ and by \eqref{106bii} of Theorem~\ref{106} above, the fiber of $\tilde\A$ over $x$ consists of two points.
In each of these pictures there are four nodes of valency three. 
In the left hand picture, each lies in a $1$-dimensional stratum in $\tilde\A$ where the local representation is of type $A_2$, so that locally $\A$ consists of three planes in $3$-space, meeting along a common line. 
The preimage of this stratum in $\tilde\A$ is a line, along which $\tilde\A$ is locally isomorphic to the union of the three planes $\ideal{e_1,e_4}, \ideal{e_2,e_4}$ and $\ideal{e_3,e_4}$ in $4$-space.

\end{asparaenum}

It would be interesting to find explicit embeddings of the space $\tilde\A$ in the remaining cases. 

\end{exa}

To prove the theorem, beginning with the multiplicative structure on $\Der_B$ and $\coker(K)$ coming from Dubrovin's Frobenius structure, we endow both $\Der_V$ and $\coker J$ with a multiplication, and $\Der_V$ with a $\Der_B$-module structure, whose crucial feature is that the derivative $tp\colon\Der_V\to\Der_B\otimes_{\OO_B}\OO_V$ of $p$ is $\Der_B$-linear.
On $\Der_V$, but not on $\coker J$, this multiplication lacks a neutral element. 

Nevertheless, the first evidence for the theorem was found by an entirely different route not involving Dubrovin's Frobenius structure.
This was based on the fact that the cokernel of the linear map $\xymatrix{S^\ell\ar[r]^\Lambda& S^\ell}$ defined by a square matrix $\Lambda$ has a natural $S$-algebra structure if and only if the so-called {\em rank condition} (rc) holds.
This is a purely algebraic condition on the adjugate matrix of $\Lambda$, which can be checked by explicit calculation. 
We explain this in general in Section~\ref{97}.

In Section~\ref{230}, we then specialize to the case where $\Lambda$ is the Jacobian matrix $J$ of the basic invariants of a Coxeter group $\A$,  or the Saito matrix of the discriminant $D$ of a Coxeter group.
The space $\tilde D=\Spec\coker K$ is normal (indeed smooth) exactly in the $ADE$-case; on the other hand $\tilde\A=\Spec\coker J$ is normal only in the case of $A_1$. 
We discuss the geometry of these two spaces, and their link with the representation theory. 
In particular we compare them with the normalizations of $D$ and $\A$ in Subsection~\ref{49}.  

In Section~\ref{7}, our earlier approach to the main theorem lead to an interesting problem on Coxeter groups.
The algebra of the fiber over $0$ of the projection $p\colon V\to V/W$ carries two structures: that of a zero-dimensional Gorenstein algebra and that of the regular $W$-representation.
It is not clear how these two structures are related: 
which irreducible components of the same $W$-isomorphism type admit an isomorphism induced by the algebra structure? 
The following consequence of Theorem~\ref{106}, whose proof is completed by Proposition~\ref{119}, answers this question in a special case. 

\begin{cor}\label{110}
Let $W$ be an irreducible Coxeter group in $\GL(V)$ with homogeneous basic invariants $p_1,\dots,p_\ell$, ordered by increasing degree, and let $F$ be the ideal in $\CC[V]$ generated by $p_1,\dots,p_\ell$.
Then for each $j=1,\dots,\ell$, there exists an $\ell\times\ell$-matrix $A_j$ with entries in $\CC[V]$ such that 
\[
\left(\frac{\p p_\ell}{\p x_1}, \dots, \frac{\p p_\ell}{\p x_\ell}\right)=\left(\frac{\p p_j}{\p x_1}, \dots, \frac{\p p_j}{\p x_\ell}\right) A_j\mod F\cdot(\CC[V])^\ell.
\]
\end{cor} 

In all cases except for $E_6$, $E_7$ and $E_8$, we give an explicit formula for the matrices $A_j$ in Corollary~\ref{110}: they are Hessians of basic invariants.
This statement is a $3$rd order partial differential condition on the basic invariants which we call the \emph{Hessian rank condition} (Hrc).
Besides the missing proof for the $E$-types, which would lead to a self contained algebraic proof of Theorem~\ref{106}, it would be interesting to know whether (Hrc) is a new condition or can be explained in the framework of Frobenius manifolds.

In our final Section~\ref{96}, we show that by adding to $D$ a divisor which pulls back to the conductor of the ring extension $\OO_D\to\coker K$, we obtain a new free divisor (Theorem~\ref{68}). 
This was already shown on the singularity side in \cite{MS10}. 
The preimage in $V$ of this free divisor is a free divisor containing the reflection arrangement (Corollary~\ref{55}). 

\subsection*{Acknowledgments}

We thank the ``Mathematisches Forschungsinstitut Oberwolfach'' for two two-week ``Research in Pairs'' stays in 2010 and 2011. 
The authors are grateful to the referee for forceful, detailed and helpful comments on an earlier version.

\section{Review of Coxeter groups}\label{80}

For more details on the material reviewed in this section, we refer to the book of Humphreys~\cite{Hum90}.
Let $V_\RR$ be an $\ell$-dimensional $\RR$-vector space and let $V=V_\RR\otimes_\RR\CC$.
Consider a finite group $W\subset\GL(V)$ generated by reflections defined over $\RR$.
Any such representation $W$ decomposes into a direct sum of irreducible representations, and $W$ is irreducible if and only if the corresponding root system is.
The irreducible isomorphism types are $A_\ell$, $B_\ell$, $D_\ell$, $E_6$, $E_7$, $E_8$, $F_4$, $G_2=I_2(6)$, $H_3$, $H_4$, and $I_2(k)$.

The group $W$ acts naturally on the symmetric algebra $S:=\CC[V]$ by the contragredient action, and we denote by $R:=S^W$ the corresponding graded ring of invariants.
By a choice of linear basis, we identify $S$ with $\CC[x_1,\dots,x_\ell]$.
The natural inclusion $R\subset S$ turns $S$ into a finite $R$-module of rank $\#W$. 
The averaging operator
\begin{equation}\label{17}
\#\colon S\to R,\quad g\mapsto g^\#:=\frac{1}{\#W}\sum_{w\in W}g^w
\end{equation}
defines a section of this inclusion.

By Chevalley's theorem (\cite[Thm. 3.5]{Hum90}), $R$ is a polynomial algebra $R=\CC[p_1,\dots,p_\ell]$ where
$p_1,\dots,p_\ell$ are homogeneous $W$-invariant polynomials in $S$.
We set 
\begin{equation}\label{84}
\deg p_i=m_i+1=w_i
\end{equation}
and assume that $m_1\le\cdots\le m_\ell$.
Then the \emph{degrees} $w_i$, or the \emph{exponents} $m_i$, are uniquely determined and 
\begin{equation}\label{1}
\sum_{i=1}^\ell m_i=\#\A
\end{equation}
where $\A$ is the arrangement of reflection hyperplanes of $W$ (\cite[Thm. 3.9]{Hum90}).

We make this more precise in the case $W$ is irreducible.
Then the eigenvalues of any Coxeter element are $\exp(2\pi i \frac{m_i}{h})$ where $h$ is the Coxeter number (\cite[Thm. 3.19]{Hum90}).
Moreover, 
\begin{gather}
\label{30}1=m_1<m_2\le\cdots\le m_{\ell-1}<m_\ell=h-1,\\
\label{83}m_i+m_{\ell-i+1}=h.
\end{gather}
In particular, this implies that $\sum_{i=1}^{\ell}m_i=\frac{\ell h}{2}$.
For $m_1=1$, the $W$-invariant $2$-form $p_1$ is unique up to a constant factor. 
By a choice of a positive multiple of $p_1$, it determines a unique $W$-invariant Euclidean inner product $(\cdot,\cdot)$ on $V_\RR$, which turns $W$ into a subgroup of $\OG(V_\RR)$ and serves to identify $V_\RR$ and $V_\RR^*$. 
With respect to dual bases of $V_\RR$ and $V_\RR^*$ we notice that the two corresponding inner products have mutually inverse matrices. 
At the level of $V^*$, we denote by 
\[
\Gamma:=((x_i,x_j))=((dx_i,dx_j))
\]
the (symmetric) matrix of $(\cdot,\cdot)$ with respect to coordinates $x_1,\dots,x_\ell$.
In suitable coordinates
\begin{equation}\label{4}
p_1=\sum_{i=1}^\ell x_i^2,\quad(x,y)=\sum_{i=1}^\ell x_iy_i,\quad\Gamma=(\delta_{i,j}).
\end{equation}
We refer to such coordinates as standard coordinates. 
In case $W$ is reducible, we have the above situation on each of the irreducible summands separately.

Geometrically the finiteness of $S$ over $R$ means that the map 
\begin{equation}\label{39}
\xymat{
V=\Spec S\ar[r]^p&\Spec R=V/W
}
\end{equation}
is finite of degree $\#W$.
We identify the reflection arrangement $\A$ of $W$ with its underlying variety $\bigcup_{H\in\A}H$.
Let $\Delta$ be a reduced defining equation for $\A$, and denote by $D=p(\A)$ the discriminant.
An {\em anti-invariant} of $W$ is a relative invariant $f\in S$ with associated character $\det^{-1}$, that is, $wf=\det^{-1}(w)f$ for all $w\in W$.
The following crucial fact due to Solomon~\cite[\S3, Lem.]{Sol63} (see also (\cite[Prop. 3.13(b)]{Hum90}) implies that $\Delta^2$ is a reduced defining equation for $D$. 

\begin{thm}[Solomon]\label{8}
$R\Delta$ is the set of all anti-invariants.\qed
\end{thm}

A second fundamental fact, due to K. Saito~\cite[\S3]{Sai93}, is the following

\begin{thm}[Saito]\label{9}
For irreducible $W$, $\Delta^2$ is a monic polynomial in $p_\ell$ of degree $\ell$, that is,
\[\pushQED{\qed}
\Delta^2=\sum_{k=0}^\ell a_{\ell-k}(p_1,\dots,p_{\ell-1})p_\ell^k,\ \text{ with }\quad a_0=1.\qedhere
\]
\end{thm}

We denote by $\Der_S$ and $\Der_R$ the modules of vector fields on $V=\Spec S$ and $V/W=\Spec R$ respectively. 
The group $W$ acts naturally on $\Der_S$.
Terao~\cite{Ter83} showed that each $\theta\in \Der(-\log D)$ has a unique lifting $p^{-1}(\theta)$ to $V$ and that the set of lifted vector fields  is 
\[
p^{-1}\Der(-\log D)=(\Der_S)^W,\quad p^*\Der(-\log D)=(\Der_S)^W\otimes_RS=\Der(-\log\A),
\]
and both $\A$ and $D$ are free divisors.
This can be seen as follows:
We denote by 
\begin{equation}\label{100}
J:=(\p_{x_j}(p_i))
\end{equation}
the Jacobian matrix of $p$ in \eqref{39} with respect to the coordinates $x_1,\dots,x_\ell$ and $p_1,\dots,p_\ell$. 
Via the identification of the 1-form $dp_i$ with a vector field $\eta_i$ such that $(dp_i,-)=\ideal{\eta_i,-}$,
\begin{align}\label{40}
dp_i=\sum_{j=1}^\ell\p_{x_j}(p_i)dx_j\leftrightarrow\eta_i&=\sum_{j=1}^\ell\ideal{\eta_i,dx_j}\p_{x_j}=\sum_{j=1}^\ell(dp_i,dx_j)\p_{x_j}\\
\nonumber&=\sum_{k,j=1}^\ell\p_{x_k}(p_i)(dx_k,dx_j)\p_{x_j}=\sum_{k,j=1}^\ell\p_{x_k}(p_i)(x_k,x_j)\p_{x_j},
\end{align}
the basic invariants define invariant vector fields $\eta_1,\dots,\eta_\ell\in(\Der_S)^W$, which must then be in $\Der(-\log\A)$.
By \eqref{40}, their Saito matrix reads
\begin{equation}\label{2}
(\eta_j(x_i))=\Gamma J^t
\end{equation}
Now $\det J$ is an anti-invariant because $J$ is the differential of the invariant map $p=(p_1,\dots,p_\ell)$. 
Hence, $\det J\in\CC^*\Delta$ by Theorem~\ref{8}, \eqref{1}, and the algebraic independence of the $p_i$.
By scaling $p$, we can therefore assume that 
\begin{equation}\label{19}
\det J=\Delta.
\end{equation}
Saito's criterion (\cite[(1.8) Thm.~ii)]{Sai80}) then shows that $\A$ is free with basis $\eta_1,\dots,\eta_\ell$.
Applying the tangent map $tp$ (see \eqref{197}) gives vector fields $\delta_1,\dots,\delta_\ell\in\Der_R$ such that $\delta_j\circ p=tp(\eta_j)$ with (symmetric) Saito matrix
\begin{equation}\label{3}
K=(K^i_j):=(\delta_j(p_i))=J\Gamma J^t
\end{equation}
with $\det(J\Gamma J^t)\in\CC^*\Delta^2$.
At generic points of $\A$, $p$ is a fold map and hence 
\begin{equation}\label{32}
\delta_1,\dots,\delta_\ell\in\Der(-\log D). 
\end{equation}
Again Saito's criterion shows that $D$ is a free divisor with basis $\delta_1,\dots,\delta_\ell$.
In standard coordinates as in \eqref{4}, this proves

\begin{lem}\label{191}
$D$ admits a symmetric Saito matrix $K=JJ^t$.
\end{lem}

If $W$ is irreducible then, in standard coordinates as in \eqref{4}, 
\begin{equation}\label{31}
\chi_w:=\frac12\delta_1=\sum_{i=1}^\ell w_ip_i\p_{p_i}.
\end{equation}
We shall refer to the grading defined by this semisimple operator as the $w$-grading.
In particular, $\delta_k$ is $w$-homogeneous of degree $w_k-w_1$.
If $W$ is reducible, we have a homogeneity such as \eqref{31} for each irreducible summand.

Throughout the paper we will abbreviate 
\[
S_\A:=S/S\Delta,\quad R_D:=R/R\Delta^2.
\]

\section{F-manifold-structures}\label{190}

In this section we prove Theorem~\ref{106}. 
We will make use of the Frobenius manifold structure on $V/W$, constructed by Dubrovin in \cite{Dub98}. 
However our main reference for background on Frobenius manifolds (including this result) is the book of Hertling \cite{Her02}. 
In fact the only aspects of the Frobenius structure we use are the existence of an integrable structure of commutative associative $\CC$-algebras on the fibers of the tangent bundle; a manifold with this structure is called by Hertling and Manin an \emph{F-manifold}. 
This notion is much simpler than
that of Frobenius manifold, omitting as it does all of the metric properties, and the connections, which make the definition of Frobenius manifold so complicated. 
Following Hertling, we use local analytic methods, and in particular local analytic coordinate changes, in
order to make use of normal forms.
Such analytic methods will be justified in Remark~\ref{212}, and we pass to the analytic category without changing our notation.

The following account summarizes parts of \cite[Ch.~2]{Her02}. 
For any $n$-dimensional F-manifold $M$, the multiplication on $TM$ is encoded by an $n$-dimensional subvariety of $T^*M$, the {\em analytic spectrum} $L$, as follows: for each point $p\in M$, points in $T^*_pM$  determine $\CC$-linear maps $T_pM\to\CC$; among these, a finite number are $\CC$- algebra homomorphisms. 
These finitely many points in each fiber of $T^*M$ piece together to form $L$. 
The composite 
\beq\label{192}
\Der_M\to\pi_*\OO_{T^*M}\to\pi_*\OO_L
\eeq
is in fact an isomorphism of $\CC$-algebras (\cite[Thm.~2.3]{Her02}).
 
The multiplication $\circ$ in $TM$ satisfies the integrability property 
\[
\Lie_{X\circ Y}(\circ)=X\circ\Lie_Y(\circ)+Y\circ\Lie_X(\circ).
\]
Provided the multiplication is generically semi-simple, as is the case for the structure constructed by Dubrovin and Hertling, this implies that $L$ is Lagrangian (\cite[Theorem 3.2]{Her02}). 
This in turn means that the restriction to $L$ of the canonical action form $\alpha$ on $T^*M$ is closed and therefore exact.
A \emph{generating function} for $L$ is any function $F\in\OO_L$ such that $dF=\alpha|_L$.
A generating function determines an \emph{Euler field} $E$ on $M$, namely a vector field mapped to $F$ by the isomorphism \eqref{192}. 
The discriminant of $M$ is defined by any of the following equivalent characterizations:
\ben
\item $D=\pi(F^{-1}(0))$,
\item $D$ is the set of points $x\in M$ where the endomorphism $E\circ\colon T_xM\to T_xM$ is not invertible.
\een
Similarly, the module $\Der(-\log D)$ may be viewed as either
\ben
\item the set of vector fields whose image under the isomorphism \eqref{192} vanishes on $F^{-1}(0)$, or equivalently as
\item the image in $\Der_M$ of multiplication by $E$. 
\een
This yields the well-known

\begin{lem}\label{194}
The discriminant $D$ is a free divisor, and the cokernel $\tilde\OO_D=\coker K$ of the Saito matrix $K$ of $D$ acquires an $\OO_M$-algebra structure as quotient of the Frobenius manifold multiplication in $\Der_M$.
\end{lem}

\begin{proof} 
The matrix of multiplication by $E$ with respect to the basis $\p_{x_1},\dots,\p_{x_\ell}$ of $\Der_M$ is $K$. 
Thus
\beq\label{209}
\xymat{
0\ar[r]&\OO_M^\ell\ar[d]^\cong\ar[r]^-K&\OO_M^\ell\ar[d]^\cong\ar[r]&\tilde\OO_D\ar[d]^\cong\ar[r]&0\\
0\ar[r]&\Der_M\ar[r]^-{E\circ}&\Der_M\ar[r]&\Der_M/\Der_M(-\log D)\ar[r]&0}
\eeq
is a presentation of $\Der_M/E\circ\Der_M=\Der_M/\Der_M(-\log D)$, which is itself isomorphic to $\pi_*\OO_{F^{-1}(0)}$. 
\end{proof}

We now return to the context and notation of Section \ref{80}. 
Our $F$-manifold is $M=V/W$, and instead of the sheaf $\OO_M$ we consider the algebra $R=\CC[V/W]$ of its global sections.
Based on Lemma~\ref{194}, we define the space $\tilde D:=\Spec\tilde R_D$.

Recall from \eqref{100} that $J\colon S^\ell\to S^\ell$ is the matrix of the morphism
\beq\label{197}
tp:\Der_S\to p^*\Der_R=\Der_R\otimes_RS,\quad 
tp(\sum_{j=1}^n\eta_j\p_{x_j})=\sum_{i=1}^n\sum_{j=1}^n\eta_j\p_{x_j}(p_i)\p_{p_i},
\eeq
defined by left composition (of vector fields as sections of $TV$) with $dp$. 
The following diagram, in which the vertical arrows are bundle projections,  helps
to keep track of these morphisms. 
Sections of $p^*\Der_R$ are maps from bottom left to top right making the lower triangle in the diagram commute.
\beq\label{std}
\xymatrix{
TV\ar[r]^{dp}\ar[d]&T(V/W)\ar[d]\\
V\ar[r]_p&V/W}
\eeq

Both $tp:Der_S\to p^*\Der_R$ and  and $\omega p:\Der_R\to p^*\Der_R$, defined by right composition with $p$, are familiar in singularity theory.
By definition, 
\beq\label{201}
\chi\in\Der_R\text{ lifts to }\eta\in\Der_S\quad\iff\quad tp(\eta)=\omega p(\chi).
\eeq

Using Lemma~\ref{191}, \eqref{209}, and the obvious identifications, there is a commutative diagram of $S$-modules
\beq\label{208}
\xymat{&J_\Delta\\
0\ar[r]&\Der_S\ar[r]^-{tp}\ar @{->>}[u]&\Der_R\otimes_RS\ar[r]&\tilde S_\A\ar[r]&0\\
0\ar[r]&R^\ell\otimes_RS\ar[u]_-{J^t}\ar[r]^-{K\otimes 1}&\Der_R\otimes_RS\ar[r]\ar[u]_=&\tilde R_D\otimes_RS\ar[r]\ar @{->>}[u]&0
}.
\eeq
Both rows here are exact: the upper row defines $\tilde S_\A$, and the lower row is the tensor product with the flat $R$-module $S$ of the short exact sequence defining $\tilde R_D$. 
Now $\tilde R_D\otimes_RS$, as a tensor product of rings, has a natural ring structure; to show that $\tilde S_{\A}$ is a ring, it will be enough to show

\begin{lem}\label{211}
The image of $tp$ is an ideal of $\Der_R\otimes_RS$.
\end{lem}

We prove Lemma~\ref{211} by showing that the Frobenius multiplication in $\Der_R$ lifts to a $p^*\Der_R$-module structure on $\Der_S$, and that $tp:\Der_S\to\Der_R\otimes_RS$ is $\Der_R$-linear.  

\begin{prp}\label{200}\
\ben
\item\label{200a} The Frobenius multiplication in $\Der_R$ can be lifted to $\Der_S$, though without multiplicative unit.
\item\label{200b} The same procedure makes $\Der_S$ into a $\Der_R$-module. 
\item\label{200c} The map $tp$ in \eqref{197} is $\Der_R$-linear, with respect to the structure in \eqref{200b} and Frobenius multiplication induced on $\Der_R\otimes_RS$.
\een
\end{prp}

\begin{proof}
By \eqref{201}, for a multiplication in $\Der_S$, \eqref{200a} means that
\beq\label{202}
tp(\eta_1\circ\eta_2)=\omega p(\chi_1\circ\chi_2)
\eeq
where $\eta_i\in\Der_S$ is a lift of $\chi_i\in\Der_R$ for $i=1,2$.
Similarly, the scalar multiplication for \eqref{200b} must satisfy
\beq\label{204}
tp(\chi\cdot\eta)=\omega p(\chi\circ\xi)
\eeq
where $\chi\in\Der_R$ and $\eta\in\Der_S$ is a lift of $\xi\in\Der_R$.

Locally, at a point $v\in V\setminus\A$, $p$, $tp$ and $\omega p$ are isomorphisms, so there is nothing to prove.
Now suppose $v\in H$ is a generic point on a reflecting hyperplane $H\in\A$, with $p(v)$ outside the bifurcation set $B$. 
In a neighborhood of $p(v)$ in $V/W$, we may take canonical coordinates $u_1,\dots,u_\ell$ (cf.~\cite[2.12.(ii)]{Her02}). 
These are characterized by the property that the vector fields $e_i:=\p_{u_i}$, $i=1,\dots,\ell$ satisfy $e_i\circ e_j=\delta_{i,j}\cdot e_i$.
By \cite[Cor.~4.6]{Her02}, the tangent space $T_{p(v)}D$ is spanned by $\ell-1$ of these idempotent vector fields, and the remaining idempotent, which we label $e_1$, is normal to it. 
The map $p_v\colon(V,v)\to(V/W,p(v))$ has multiplicity $2$, critical set $H$ and set of critical values $D$, from which it follows that $d_vp\colon T_vH\to T_{p(v)}D$ is an isomorphism. 
Since we have fixed our coordinate system on $(V/W,p)$, we are free to choose only the coordinates on $(V,v)$. 
Define $x_i=u_i\circ p$ for $i=2,\dots,\ell$. 
To extend these to a coordinate system on $(V,v)$, we may take as $x_1$ any function whose derivative at $v$ is linearly independent of $d_vx_2,\dots, d_vx_\ell$. 
This means we may take as $x_1$ any defining equation of the critical set (the hyperplane $H$) of $p$ at $v$.   
With respect to these coordinates, $p$ takes the form
\beq\label{203}
p_v(x_1,\dots,x_\ell)=(f(x_1,\dots,x_\ell),x_2,\dots,x_\ell).
\eeq
As $p_v$ has critical set $\{x_1=0\}$ and discriminant $\{u_1=0\}$, both $f$ and $\p_{x_1}(f)$ vanish along $\{x_1=0\}$. 
Thus $f(x)=x_1^2g(x)$ for some $g\in\OO_{V,v}$. 
Since $p$ has multiplicity $2$ at $v$, $g(0)\neq 0$. 
Now replace the coordinate $x_1$ by $x_1g(x)^{1/2}$. 
With respect to these new coordinates, which we still call $x_1,\dots,x_\ell$, $p_v$ becomes a standard fold:
\[
p_v(x_1,\dots,x_\ell)=(x_1^2,x_2,\dots,x_\ell).
\]
We can now explicitly calculate the multiplication in $\Der_S$, locally at $v$:
\[
\begin{cases}
tp_v(x_1\p_{x_1})=\omega p_v(2u_1\p_{u_1}),\\
tp_v(\p_{x_i})=\omega p_v(\p_{u_i}),&\text{ for }i=2,\dots,\ell.
\end{cases}
\]
So \eqref{202} implies
\begin{align*}
tp_v((x_1\p_{x_1})\circ(x_1\p_{x_1}))
&=\omega p_v((2u_1\p_{u_1})\circ(2u_1\p_{u_1}))\\
&=\omega p_v(4u_1^2\p_{u_1})
=\omega p_v(2u_1(2u_1\p_{u_1}))
=tp_v((2x_1^2)x_1\p_{x_1}),
\end{align*}
and hence $x_1\p_{x_1}\circ x_1\p_{x_1}=2x_1^3\p_{x_1}$.
So in order that \eqref{202} should hold, we are forced to define
\[
\p_{x_i}\circ\p_{x_j}=
\begin{cases}
2x_1\p_{x_1},&\text{for }i=j=1,\\
\delta_{i,j}\cdot\p_{x_i},&\text{otherwise}.
\end{cases}
\]
Since the multiplication in $\Der_V$ is uniquely defined by \eqref{202} outside codimension $2$, it extends to $V$ by Hartog's Extension Theorem.
This proves \eqref{200a}; \eqref{200b} is obtained by an analogous argument using \eqref{204}.

Finally, \eqref{200c} follows from \eqref{201} and \eqref{204} on $V\setminus\A$, and therefore holds everywhere. 
\end{proof} 

\begin{proof}[Proof of Lemma~\ref{211}]
Let $\xi\in \Der_S$, $g\in S$ and $\eta\in\Der_R$. 
By Proposition~\ref{200}.\eqref{200c} and the evident $S$-linearity of
the lifted Frobenius multiplication, 
\[
(\eta\otimes_Rg)\cdot tp(\xi)=tp(\eta\circ g\xi).
\]
\end{proof}

We have proved the following result, which implies \eqref{106a} of Theorem~\ref{106}.

\begin{thm}\label{206}
The cokernel $\tilde S_\A=\coker J$ of the transposed Saito matrix of $\A$ is an $S_\A$-algebra.\qed
\end{thm}

Based on Theorem~\ref{206}, we define the space $\tilde\A:=\Spec\tilde S_\A$.

\begin{rmk}\label{212}
Even though our proof uses complex analytic methods, such as canonical coordinates in the proof of Proposition~\ref{200}, the conclusion is valid over any field over which the basic invariants are defined. We show this in Section~\ref{97} below by proving that the fact that $\coker J$ is an $S$-algebra is equivalent to a condition on ideal membership, the so-called {\em rank condition} (rc). 
\end{rmk} 

We end this section by clarifying the relationship between $\tilde\A$ and $\tilde D\times_D\A$ which are not isomorphic in general.
For $\tilde R_D\otimes _RS_\A$ is the cokernel of $1\otimes\Delta\colon\tilde R_D\otimes_RS\to\tilde R_D\otimes_RS$, and using the epimorphism $\Der_R\otimes_RS\onto\tilde R_D\otimes_RS$ we find that there is an epimorphism $\Der_R\otimes_RS\onto\tilde R_D\otimes_RS_\A$, whose kernel is equal to $\Der_R\otimes_RS\Delta+\Der(-\log D)\otimes_RS$. 
Both summands here are contained in the image of $tp\colon\Der_S\to\Der_R\otimes_RS$, the first by Cramer's rule and the second because every vector field $\eta\in\Der(-\log D)$ is liftable via $p$. 
Thus $\tilde S_\A$ is a quotient of $\tilde R_D\otimes_RS_\A$.
The kernel $N$ of the projection $\tilde R_D\otimes S_\A\to\tilde S_\A$ is the quotient
\[
N:=tp(\Der_S)/\bigl(\Der(-\log D)\otimes_RS+\Der_R\otimes_RS\Delta\bigr).
\]
At a generic point $x\in\A$ this vanishes: here $p$ is a fold map, right-left-equivalent to 
\[
(x_1,\dots,x_\ell)\mapsto (x_1,\dots,x_{\ell-1},x_\ell^2)
\]
and an easy local calculation shows that in this case $N_x=0$. 
However, if $p$  has multiplicity $>2$ at $x$ then $N_x\neq 0$. 
For example at an $A_2$ point, $p$ is right-left equivalent to  
\[
(x_1,\dots, x_\ell)\mapsto (x_1^2+x_1x_2+x_2^2,x_1x_2(x_1+x_2),x_3,\dots,x_\ell);
\]
$tp(\Der_S)$ is generated by  $\p_{p_3},\dots,\p_{p_\ell}$ together with
\[
(2x_1+x_2)\p_{p_1}+(2x_1x_2+x_2^2)\p_{p_2},(x_1+2x_2)\p_{p_1}+(x_1^2+2x_1x_2)\p_{p_2}, 
\]
while the coefficients of $\p_{p_1}$ in the generators of $\Der(-\log D)\otimes_RS+\Der_R\otimes_RS\Delta$ are at least quadratic in $x_1,\dots,x_\ell$.
In fact we have 

\begin{thm}\label{220}
$\tilde\A=(\tilde D\times_D\A)_\red$ 
\end{thm}

\begin{proof}
$\tilde S_\A=\coker tp$, with $tp$ as in \eqref{208}, is a maximal Cohen--Macaulay $S_\A$-module of rank $1$.
This means that at a smooth point of $\A$, $\tilde S_\A$ is isomorphic to $S_\A$, and is thus reduced. 
As $\tilde S_\A$ is finite over $S_\A$, its depth over itself (assuming it is a ring) is equal to its depth over $S_\A$. 
Since it is therefore a Cohen--Macaulay ring, generic reducedness implies reducedness.
\end{proof}
For later use we note that 
by \cite[Cor.~3.15]{MP89}, we have

\begin{thm}\label{105}
$\tilde\A$ is Cohen--Macaulay and $\tilde D$ is Gorenstein.\qed
\end{thm}

\section{Algebra structures on cokernels of square matrices}\label{97}

\subsection{Rank condition}

In this subsection we recall a condition on the rows of the adjugate of a square matrix over a ring $R$, which is equivalent to that matrix presenting an $R$-algebra, at least in the local and local graded cases.
It is the key to proving Corollary~\ref{110} in the Introduction.

Let $R$ be an $\ell$-dimensional (graded) local Cohen--Macaulay ring with maximal (graded) ideal $\mm$.
In the graded local case, we assume that all $R$-modules are graded and all $R$-linear maps are homogeneous.

Let $\Lambda$ be an $\ell\times\ell$-matrix over $R$ with transpose $A:=\Lambda^t$.
We consider both $\Lambda$ and $A$ as $R$-linear maps $R^\ell\to R^\ell$. 
Assume that $\Delta:=\det\Lambda$ is a reduced non-zero-divisor and set $D=V(\Delta)$.
By Cramer's rule $\Delta$ annihilates $M:=\coker\Lambda$ which is hence a module over $R_D:=R/R\Delta$. 
For any ideal $I\subseteq R$, we denote by $I_D:=R_DI$ its image in $R_D$.
By $Q_D:=Q(R_D)$, we denote the total ring of fractions of $R_D$. 

The {\em $k$-th Fitting ideal} of $M$ {\em over $R$}, $\Fit^k(M)$, is the ideal of $R$ generated by the $(\ell-k)\times(\ell-k)$-minors of $\Lambda$. 
It is an invariant of $M$, and independent of the presentation $\Lambda$. 
We denote by $m^i_j$ the generator of $\Fit^1(M)$ obtained from $\Lambda$ by deleting row $i$ and column $j$.
Note that $\Fit_D^k(M)$ is the $k$'th Fitting ideal of $M$ over $R_D$. 
For properties of Fitting ideals, see e.g.~\cite[Ch.~20]{Eis95}.

\begin{dfn}\label{78}
We say that the \emph{rank condition} (rc) holds for $\Lambda$ if $\grade\Fit^1(M)\ge 2$ and $\Fit^1(M)$ is equal to the ideal of maximal minors of the matrix obtained from $\Lambda$ by deleting one of its rows, possibly after left multiplication of $\Lambda$ by some invertible matrix over $R$.
\end{dfn}

Note that (rc) implies that $\Fit_D^1(M)$ is a maximal Cohen--Macaulay $R_D$-module, by the Hilbert--Burch theorem. 
It turns out that (rc) depends only on the module $M=\coker\Lambda$, and not on the choice of presentation $\Lambda$. This is a consequence of the following two theorems, which also make clear the reason for our interest in the condition (rc).

\begin{thm}[{\cite[Thm.~3.4]{MP89}}]\label{162}
If $M$ is an $R_D$-algebra then (rc) holds for $\Lambda$.\qed
\end{thm}

The proof in \cite{MP89} shows that if $M$ is an $R_D$-algebra by $e, m_2,\dots, m_\ell$, where $e$ is the multiplicative identity of $M$, and $\Lambda$ is a presentation of $M$ with respect to these generators, then $\Fit^1(M)$ is equal to the ideal of maximal minors of $\Lambda$ with its first row deleted. 

The converse theorem also holds. 
A proof, due to de~Jong and van~Straten, can be found in \cite[Prop.~3.14]{MP89}. 
We will use some of the notions introduced there, however, and so we give a sketch, based on the accounts there and in \cite{JS90a}.

Recall that a \emph{fractional ideal} $U$ (over $R_D$) is a finitely generated $R_D$-submodule of $Q_D$ which contains a non-zero-divisor and that 
\beq\label{112}
\Hom_{R_D}(U,V)=V\colon_{Q_D}U
\eeq
is a fractional ideal, for any two fractional ideals $U$ and $V$.
We shall use this identification implicitly.
In particular, the duality functor 
\[
(-)^\vee:=\Hom_{R_D}(-,R_D)
\]
preserves fractional ideals. 
It is inclusion reversing and a duality on maximal Cohen--Macaulay fractional ideals (see \cite[Prop.~1.7]{JS90a}). 

\begin{thm}\label{160} 
If (rc) holds for $\Lambda$ then $M$ is a fractional ideal generated over $R_D$ by $\varphi_1,\dots,\varphi_\ell\in Q_D$ where
\beq\label{131}
\varphi_i m^\ell_j= m^i_j,\quad i,j=1,\dots,\ell;
\eeq
moreover $M$ is an $R_D$-subalgebra of $Q_D$ isomorphic to $\End_{R_D}(\Fit^1_D(M))$.
\end{thm}

\begin{proof}
Using (rc) for $\Lambda$, Lemma~\ref{136} (below) yields a presentation 
\beq\label{132}
\xymat{0\ar[r]&R^{\ell}\ar[r]^-A&R^\ell\ar[rr]^-{(m^\ell_1,\dots,m^\ell_\ell)}&&F^1_D(M)\ar[r]&0}.
\eeq
In particular, $\Fit^1_D(M)$ is a maximal Cohen--Macaulay $R_D$-module of rank $1$, and therefore can be viewed as a fractional ideal. 
As $\Fit^1_D(M)$  is contained in $R_D$, $\Fit^1_D(M)^\vee$ is a fractional ideal containing $R_D$.
Dualizing \eqref{132} with respect to $R_D$ gives the exact sequence
\[
\xymat{0\ar[r]&F^1_D(M)^\vee\ar[r]&R_D^\ell\ar[r]^-\Lambda&R_D^\ell}.
\]
There is also a $2$-periodic exact sequence
\[\label{116}
\xymat{
\cdots\ar[r]&R_D^\ell\ar[r]^-\Lambda&R_D^\ell\ar[r]^-{\ad\Lambda}&R_D^\ell\ar[r]^-\Lambda&\cdots}.
\]
Therefore, 
\[
\Fit^1_D(M)^\vee\cong\ker_{R_D}\Lambda\cong
\begin{cases}
\coker_{R_D}\Lambda=M,\\
\img_{R_D}\ad\Lambda=\Fit^1_D(M).
\end{cases}
\]
and hence $\Fit^1_D(M)^\vee\cong\End_{R_D}(\Fit^1_D(M))$.
From this all the statements follow.
\end{proof}

In Subsection~\ref{117} we identify the generators in Theorem~\ref{160} in the case that $D$ is the reflection arrangement or discriminant of an irreducible Coxeter group.

\begin{lem}[{\cite[Prop.~1.10]{JS90a}}]\label{136}
Suppose that the ideal $I$ (generated by the maximal minors of the matrix $\Lambda$ with one row deleted) has grade $2$.
Then there is a free resolution  
\beq\label{134}\pushQED{\qed}
\xymat{0\ar[r]&R^\ell\ar[r]^A&R^\ell\ar[rr]^-{(m_1^\ell,\dots,m_\ell^\ell)}&&I_D\ar[r]&0}.\qedhere
\eeq
\end{lem}

We can now make good the promise we made in Remark~\ref{212}: that Theorem~\ref{206} is valid over any field $\KK$ over which the basic invariants $p_1,\dots, p_\ell$ are defined.
From Theorems~\ref{206} and \ref{162} it follows that (rc) holds for $\Lambda$ analytically: 
for each $i,j\in\{1,\dots,\ell\}$, the equation
\beq\label{213}
m^i_j=a_1m^\ell_1+\cdots+a_\ell m^\ell_\ell
\eeq
has a solution $(a_1,\dots,a_\ell)$ where the $a_i$ are germs of complex analytic functions at $0$. 
We claim that \eqref{213} has solutions with $a_i\in\KK[V]$, and hence (rc) holds for $\Lambda$ algebraically. 

To prove this claim, first note that since the $m^i_j$ are homogeneous elements of $\KK[V]$, each $a_i$ can be replaced by its graded part of degree $D_i-D_\ell$ (see \eqref{215}).
Let $\KK[V]_d\subset\KK[V]$ be the $\KK$-vector space of all polynomials of degree $d$. 
The map
\[
m^\ell:\KK[V]_{D_i-D_\ell}^\ell\to\KK[V]_{D_i},\quad m^\ell(a_1,\dots,a_\ell)=\sum_{j=1}^\ell a_jm^\ell_j,
\]
is $\KK$-linear. 
Therefore the solvability of \eqref{213} in $\KK[V]$ reduces to a simple theorem of linear algebra, which can be rephrased more abstractly as follows:
Let $\alpha\colon\KK^m\to \KK^n$ be a $\KK$-linear map, and suppose $\KK\subset\LL$ is a field extension.
Then
\[
\img(\alpha\otimes_\KK1_\LL)\cap\KK^n=\img(\alpha).
\]
We leave the proof of this to the reader.

\subsection{Rings associated to free divisors}

In this subsection we make some general observations about the algebra presented by the transpose of a Saito matrix of a free divisor. 
Let $D=V(\Delta)$ be a free divisor in $(\CC^\ell,0)$ with Saito matrix $A$.
Then we have an exact sequence
\beq\label{156}
\xymat{0\ar[r]&R^{\ell}\ar[r]^-A&R^\ell\ar[rr]^-{(\Delta_1,\dots,\Delta_\ell)}&&R_D\ar[r]&R_D/J_D\ar[r]&0},
\eeq
where $\Delta_j:=\p\Delta/\p x_j$ for $j=1,\dots,\ell$, and $J_D:=R_DJ_\Delta$ is the Jacobian ideal of $D$. 
Now assume also that $D$ is Euler homogeneous. 
By adding multiples of the Euler vector field $\chi=\delta_1$ to the remaining members $\delta_2,\dots,\delta_\ell$ of a Saito basis of $D$, we may assume that these annihilate $\Delta$. 
We shall assume that $A$ is obtained from such a basis. 
We say that $D$ satisfies (rc) if (rc) holds for $\Lambda=A^t$.
In this case, we write  
\[
\tilde R_D:=M=\coker\Lambda\subset Q_D
\]
for the ring of Theorem~\ref{160}.

It is well known that for any algebraic or analytic space $D$ satisfying Serre's condition S2, the fractional ideal $\End_{R_D}(J_D^\vee)$ is naturally contained in the integral closure of $R_D$ in $Q_D$, and the inclusion $R_D\into\End_{R_D}(J_D^\vee)$ gives a partial normalization (see for example \cite[Ch.~2, \S2; Ch.~6, \S2]{Vas98}.
Grauert and Remmert showed in \cite{GR71} (see also \cite[Ch.~6, \S5]{GR84}) that for analytic spaces, $R_D=\End_{R_D}(J_D^\vee)$ precisely at the normal points of $D$, and the analogous result for algebraic spaces was shown by Vasconcelos in \cite{Vas91}. 

\begin{prp}\label{133} 
If the free divisor $D$ satisfies (rc) then $\tilde R_D\cong\End_{R_D}(J_D)\cong\End_{R_D}(J_D^\vee)$.
\end{prp}

\begin{proof} 
First, recall the well known fact that for $j=1,\dots,\ell$, 
\beq\label{157}
m^1_j=\frac{\Delta_j}{\deg\Delta}.
\eeq
This follows from the fact that by Cramer's rule the logarithmic $1$-form $\omega_1:=\frac1\Delta\sum_{j=1}^\ell m^1_jdx_j$ satisfies 
\[
\langle\omega_1,\delta_j\rangle=
\begin{cases}
1 & \text{if }j=1,\\
0&\text{if }j=2,\dots,\ell,
\end{cases}
\]
as does $\frac{1}{\deg\Delta}\frac{d\Delta}{\Delta}$.

Next, Lemma~\ref{136} yields a presentation 
\[
\xymat{0\ar[r]&R^{\ell}\ar[r]^-{A}&R^\ell\ar[rr]^-{(m^\ell_1,\dots,m^\ell_\ell)}&&F^1_D(M)
\ar[r]&0}.
\]
This coincides with that of $J_D$ in \eqref{156}; it follows that as $R_D$-modules, $\Fit^1_D(M)$ and $J_D$ are isomorphic. 
Hence, by Theorem~\ref{160},
\[
\tilde R_D=\End_{R_D}(\Fit^1_D(M))\cong\End_{R_D}(J_D).
\]
Since $D$ is free, $J_D$ is maximal Cohen--Macaulay, and then reflexive by \cite[Prop.~(1.7) iii)]{JS90a}.
So dualizing induces an isomorphism $\End_{R_D}(J_D)\cong\End_{R_D}(J_D^\vee)$.
\end{proof}

\begin{rmk}
The map $\varphi_1\in\End_{R_D}(\Fit^1_D(M))$ described in the proof of Theorem~\ref{160} gives an explicit isomorphism $\Fit^1_D(M)\cong J_D$.
Indeed, $\varphi_1(m^\ell_j)=\frac{\Delta_j}{\deg\Delta}$ by Lemma~\ref{157}.

However the  following example, of the discriminant of the reflection group $B_3$, shows that, even under the hypotheses of Proposition~\ref{133}, it is not necessarily the case that the other generators $\varphi_i$ of $\End_{R_D}(\Fit^1_D(M))$, $i=2,\dots,\ell$, defined in \eqref{131} are isomorphisms onto their image. 

A Saito matrix for the discriminant $D$ of $B_3$ is given by
\[
A:=
\begin{pmatrix}
x & -4x^2+18y& -xy+27z \\ 
2y & xy+27z & -2y2+18xz \\
3z & 6xz & 6yz  
\end{pmatrix}=\Lambda^t.
\]
Because this satisfies (rc),
\[
\tilde I_D=\ideal{x^2y-4y^2+3xz, x^2z-3yz, xyz-9z^2},
\]
is equal to the ideal of maximal minors of $\Lambda$ with its third column deleted. 
On the other hand the ideal of maximal minors of $A$ with its second column deleted is
\[
\ideal{x^2z-3yz, xyz-9z^2}.
\]
Evidently the two ideals are not isomorphic as $R_D$-modules. 
\end{rmk}

In contrast, for irreducible free divisors we have

\begin{prp}\label{137} 
Assume that in addition to the hypotheses of Proposition~\ref{133}, $D$ is irreducible and is not isomorphic to the Cartesian product of a smooth space with a variety of dimension $<\ell-1$.
Then each of the maps $\varphi_i$ in \eqref{131} is an isomorphism onto its image. 
Let $I_i$ denote the ideal of maximal minors of $A$ with its $i$'th row deleted.
Then, for each $i=1,\dots,\ell$, $R/I_i=R_D/I_iR_D$ is a Cohen--Macaulay ring with support $D_{\Sing}$.
\end{prp}

\begin{proof} 
Because $\Delta\in I_i$, the $(\ell-1)$-dimensional components of $V(I_i)$ are among the components of $D$. 
Since $\Delta$ is irreducible, the only component possible is $D$ itself. 
But then because $D$ is reduced, we would have $I_i\subset\ideal{\Delta}$. 
This is absurd, for by hypothesis all entries of $A$ lie in the maximal ideal, and $\Delta=\sum_{j=1}^\ell A^i_jm^i_j$.  
Thus $V(I_i)$ is purely $\ell-2$-dimensional. 
From this the result now follows by Lemma~\ref{136}.
\end{proof}

Our Propositions~\ref{133} and \ref{137} are closely related to \cite[Prop.~6.15]{Vas98}:

\begin{prp} 
If $D$ is a free divisor, then 
\beq\label{138}
J_D\cdot\Hom_{R_D}(J_D,R_D)=\Fit^1_D(M).
\eeq
Here both ideals $J_D$ and $\Hom_{R_D}(J_D,R_D)$ are viewed as fractional ideals in $Q_D$.\qed 
\end{prp}

The left hand side of \eqref{138} is the so-called \emph{trace ideal} of $J_D$; it is the set
\[
\{\varphi(g)\mid\varphi\in\Hom_{R_D}(J_D, R_D), g\in J_D\}.
\]

Buchweitz, Ebeling and Graf von Bothmer give a criterion under which, for a free divisor $D$ appearing as the discriminant in the base-space of a versal deformation of a singularity, the ring $\End_{R_D}(J_D)$ coincides with the normalization $\bar R_D$ of $R_D$:

\begin{prp}[{\cite[Thm.~2.5, Rmk.~2.6]{BEG09}}]
If $D\subset S$ is the discriminant in the smooth base-space of a versal deformation $f\colon X\to S$ and the module of $f$-liftable vector fields in $\Der_S$ is free, then provided $\codim_Sf(X_{\Sing})\geq 2$, this module coincides with $\Der(-\log D)$. 
If in fact $\codim_Sf(X_{\Sing})\geq 3$, then $\End_{R_D}(J_D)=\bar R_D$.
\end{prp}

\section{Ring structures associated with Coxeter groups}\label{230}

\subsection{Rank conditions and associated rings}\label{231}

We return to the situation of Section~\ref{80}.
From now on we work in standard coordinates as in \eqref{4}.
Denote by $J_\Delta\subset S$ and $J_{\Delta^2}\subset R$ the gradient ideals of $\Delta$ and $\Delta^2$ respectively, and by
\[
J_\A:=J_\Delta S_\A,\quad J_D:=J_{\Delta^2}R_D
\]
the Jacobian ideals of $\A$ and of $D$ respectively.
Consider the corresponding $1$st Fitting ideals
\begin{equation}\label{71}
I_{\A}:=\Fit_S^1(J_\A),\quad \tilde I_\A:=\Fit_{S_\A}^1(J_\A)=I_\A\cdot S_\A,\quad I_D:=\Fit^1_R(J_D),\quad \tilde I_D:=\Fit^1_{R_D}(J_D)=I_D\cdot R_D.
\end{equation}
By \eqref{4}, \eqref{2} and \eqref{3}, we have exact sequences
\begin{gather}\label{26}
\xymat{0\ar[r]&S^\ell\ar[r]^-{J^t}&S^\ell\ar[r]&J_\A\ar[r]&0},\\
\nonumber\xymat{0\ar[r]&R^\ell\ar[r]^-{K=JJ^t}&R^\ell\ar[r]&J_D\ar[r]&0}.
\end{gather}
The above Fitting ideals $I_\A$ and $I_D$ are generated by the sub-maximal minors of $J$ and $K$ respectively. 
Being Saito matrices, $J^t$ and $K$ have rank $\ell-1$ at smooth points of $\A$ and $D$ respectively. 
Therefore $I_\A$ and $I_D$ are ideals of grade $2$ and $\tilde I_\A$ and $\tilde I_D$ are ideals of grade $1$.

A more precise version of the rank condition (rc) from Definition~\ref{78} holds for $\A$ and $D$: 
\begin{lem}\label{161}
For irreducible $W$, $I_\A$ is generated by the maximal minors of the matrix obtained from $J$ by omitting its $\ell$'th row. 
This is its homogeneous part of minimal degree $\sum_{i<\ell}m_i=\frac{h\ell}{2}-h+1$.
\end{lem}

\begin{proof}
By a theorem of Solomon~\cite[Thm.~2, Cor.~(2a)]{Sol64} the minors of $J$ are linearly independent over $\CC$. 
As $I_\A$ is generated by $\ell$ minors, these must then be the minors of lowest degree.
\end{proof}
 
\begin{dfn}
For irreducible $W$, we refer to the condition defined in Lemma~\ref{161} as the \emph{graded rank condition} (grc) for $\A$.
Analogously, we say that the (grc) holds for $D$ if $I_D$ is generated by the entries in the $\ell$'th row of $\ad(K)$, once again the maximal minors of the matrix obtained by omitting from $K$ the highest weight vector field $\delta_\ell$.
For reducible $W$, we define (grc) for both $\A$ and $D$ by requiring it, as just defined, for each irreducible summand.
\end{dfn}
In dimension $\ell=2$, (grc) holds trivially for $\A$ and $D$: $I_\A$ and $I_D$ are the graded maximal ideals of $S_\A$ and $R_D$, due to the presence in each case of an Euler vector field.
We shall look at this case in more detail in Subsection~\ref{49}.

By Lemma~\ref{136}, (rc) for $\A$ and $D$ yields exact sequences
\begin{gather}\label{85}
\xymat{0\ar[r]&S^\ell\ar[r]^-{J^t}&S^\ell\ar[r]&\tilde I_\A\ar[r]&0},\\
\nonumber\xymat{0\ar[r]&R^\ell\ar[r]^-K&R^\ell\ar[r]&\tilde I_D\ar[r]&0}.
\end{gather}
The cokernels of the dual maps $J\in\End_S(S^{\ell})$, $K^t=K\in\End_R(R^\ell)$ are the algebras 
\begin{equation}\label{18}
\tilde S_\A=\End_{S_\A}(\tilde I_\A),\quad\tilde R_D=\End_{R_D}(\tilde I_D),
\end{equation}
of Theorem~\ref{206} and of Lemma~\ref{194}, respectively.
Recall that we write $\tilde\A=\Spec \tilde S_\A$ and $\tilde D=\Spec\tilde R_D$. 

\begin{exa}\label{46}
Let $\A$ be the reflection arrangement for $W$ of type $A_1\times\cdots\times A_1$. 
In suitable coordinates this is a normal crossing divisor defined by $\Delta=x_1\cdots x_\ell$.
Then $J=J^t=\diag(x_1,\dots,x_\ell)$ and 
\[
\tilde S_\A=\coker\,J=\CC[x_2,\dots,x_\ell]\oplus\CC[x_1,x_3,\dots,x_\ell]\oplus\cdots\oplus\CC[x_1,\dots,x_{\ell-1}].
\]
\end{exa}

Generalizing this example we have

\begin{lem}\label{28}
The assignments $W\mapsto\tilde S_\A$ and $W\mapsto\tilde R_D$ commute with direct sums (of representations/rings).
\end{lem}

\begin{proof}
Assume that $W=W'\oplus W''$, and use the analogous notation to refer to the above defined objects with $W$ replaced by $W'$ and $W''$ respectively.
Then $S=S'\otimes_{\CC} S''$, $J$ is a block matrix with blocks $J'$ and $J''$, $\Delta=\Delta'\Delta''$, hence $I_\A=I_{\A'}\Delta''+I_{\A''}\Delta'$ and therefore 
\[
\tilde I_\A\cong \tilde I_{\A'}\otimes_\CC S''\oplus S'\otimes_\CC \tilde I_{\A''}
\]
by the following Lemma~\ref{20}.
Applying $\End_{S_\A}$ yields
\[
\tilde S_\A=\tilde S_{\A'}\otimes_\CC S''\oplus S'\otimes_\CC\tilde S_{\A''}.
\]
This proves the claim for $\A$; an analogous proof works for $D$.
\end{proof}

\begin{lem}\label{20}
Let $f\in K[x]=K[x_1,\dots,x_r]\supset I$, $g\in K[y]=K[y_1,\dots,y_s]\supset J$, and $K[x,y]=K[x_1,\dots,x_r,y_1,\dots,y_s]$.
Then 
\begin{align*}
(Ig+Jf)(K[x,y]/\ideal{fg})&\cong I(K[x]/\ideal{f})\otimes_KK[y]\oplus K[x]\otimes_KJ(K[y]/\ideal{g}),\\
[Pg+Qf]&\leftrightarrow [P]\oplus[Q].
\end{align*}
\end{lem}

\begin{proof}
One easily verifies that the given correspondence is well-defined in both directions.
\end{proof}

\subsection{Relation of rings for $\A$ and $D$}\label{117}

Let us assume now that $W$ is irreducible.
Then the algebras $\tilde S_\A$ and $\tilde R_D$  can be described more explicitly as follows.
We denote by
\begin{equation}\label{14}
(m^i_j):=\ad(J^t),\quad (M^i_j):=\ad(K)=\ad(J^t)\ad(J)
\end{equation}
the adjoint matrices of $J^t$ and $K$ respectively, and set
\beq\label{215}
D_k=\deg(m^k_j)=\sum_{i=1}^\ell m_i-m_k.
\eeq
Abbreviating $h_i:=\varphi^\A_i\in Q_A$ and $g_i:=\varphi^D_i\in Q_D$ for $i=1,\dots,\ell$, Theorem~\ref{160} reads
\begin{gather}
\label{13}
h_im^\ell_j=m^i_j,\quad g_iM^\ell_j=M^i_j,\quad i,j=1,\dots,\ell,\\
\label{15}
\tilde S_\A=\ideal{h_1,\dots,h_\ell}_{S_\A}=S_\A[h_1,\dots,h_{\ell-1}],\quad \tilde R_D=\ideal{g_1,\dots,g_\ell}_{R_D}=R_D[g_1,\dots,g_{\ell-1}].
\end{gather}

\begin{prp}\label{12}
If $W$ is irreducible then 
\[
h_i=\frac{\p_{p_i}(\Delta^2)}{\p_{p_\ell}(\Delta^2)}\in Q_\A^W.
\]
\end{prp}

\begin{proof}
First, differentiate $\Delta^2\in R$,
\[
2\Delta d\Delta=d(\Delta^2)=\sum_{k=1}^\ell\p_{p_k}(\Delta^2)dp_k
\]
considered as an equality in $\Omega^1_S$.
Then wedging with $dp_1\wedge\dots\wedge\widehat{dp_i}\wedge\dots\wedge dp_{\ell-1}$ gives
\[
(-1)^{i-1}\p_{p_i}(\Delta^2)dp_1\wedge\dots\wedge dp_{\ell-1}+
(-1)^{\ell-1}\p_{p_\ell}(\Delta^2)dp_1\wedge\dots\wedge\widehat{dp_i}\wedge\dots\wedge dp_\ell\equiv 0 \mod S\Delta.
\]
Taking coefficients with respect to $dx_1\wedge\dots\wedge\widehat{dx_j}\wedge\dots\wedge dx_\ell$ yields
\[
\p_{p_i}(\Delta^2)m^\ell_j\equiv\p_{p_\ell}(\Delta^2)m^i_j\mod S\Delta,\quad j=1,\dots,\ell.
\]
By Theorem~\ref{9}, $\p_{p_\ell}(\Delta^2)$ is a non-zero-divisor in $S_\A$, and the claim follows from \eqref{13}.
\end{proof}

Using Theorem~\ref{8}, one verifies that the averaging operator \eqref{17} induces a commutative diagram of $R$-modules
\[
\xymat{
Q_D\ar@{^(->}[r] & Q_\A\ar[r]^-\# & Q_\A^W & Q_D\ar^-\cong[l] & \\
\tilde R_D\ar@{^(-->}[r]\ar@{^(->}[u] &\tilde S_\A\ar@{-->}[r]^-\#\ar@{^(->}[u] & \tilde S_\A^W\ar@{^(->}[u] & \tilde R_D\ar@{-->}[l]^-\cong\ar@{_(->}[u]\\
R_D\ar@{^(->}[u]\ar@{^(->}[r] & S_\A\ar[r]^-\#\ar@{^(->}[u] & (S_\A)^W\ar@{^(->}[u] & R_D\ar@{^(->}[u]\ar[l]^-\cong
}
\]
where the dashed maps result from the following proposition.

\begin{prp}\label{27}
We have
\begin{equation}\label{25}
h_i=g_i=\frac{M^i_\ell}{M^\ell_\ell}\in Q_D,
\end{equation} 
and hence
\begin{equation}\label{42}
\tilde R_D=(\tilde S_\A)^W.
\end{equation}
\end{prp}

\begin{proof}
Using \eqref{14} we have $M^i_\ell=\sum_rm^i_rm^\ell_r$.
By \eqref{13}, this is equal to $h_i\sum_rm^\ell_rm^\ell_r$ and therefore to $h_iM^\ell_\ell.$
By \cite[Thm.~3.4]{MP89}, $M^\ell_\ell$ generates the conductor of $R_D\into\tilde R_D$ and is therefore not a zero-divisor on $R_D$ or $S_{\A}$. 
Therefore, $h_i=M^i_\ell/M^\ell_\ell=g_i$ by \eqref{13} and \eqref{42} follows using \eqref{15}.
\end{proof}

\subsection{Local trivialization}

The integral varieties of $\Der(-\log\A)$ and $\Der(-\log D)$ form Saito's logarithmic stratification defined in \cite[\S3]{Sai80}, which we denote by $L(\A)$ and $L(D)$ respectively.
We shall locally trivialize $\tilde\A$ and $\tilde D$ along logarithmic strata with slices of the same type, with $W$ replaced by the subgroup fixing the strata.
In the case of $\tilde\A$ the trivialization is algebraic, while in the case of $\tilde D$ we need to work in the analytic category.

We begin with the discussion of $\tilde\A$.
The logarithmic stratification $L(\A)$ coincides, up to taking the closure of strata, with the intersection lattice of $\A$.
It is a geometric lattice (ordered by reverse inclusion) whose rank function is given by the codimension in $V$.
By $L_k(\A)\subset L(\A)$, we denote the collection of all rank $k$ elements.

\begin{dfn}\label{183}
For $X\in L(\A)$, denote by $W_X$ the subgroup of $W$ generated by reflections with reflecting hyperplanes in the localization $\A_X:=\{H\in\A\mid X\subset H\}$ of $\A$ along $X\in L(\A)$, and by $\Delta_X$ the reduced defining equation of $\A_X$. 
We denote also by $I_X$ the defining ideal of $X$ in $S_\A$.
For $x\in V$, let $X(x)$ be the stratum $X\in L(\A)$ with $x\in X$.
\end{dfn}

By \cite[Thm.~1.12~(d)]{Hum90}, $W_X$ is the group fixing $X$ point-wise, that is
\[
W_X=\bigcap_{x\in X}W_x.
\]
It follows that 
\[
W_{X(x)}=W_x
\]
is the isotropy group of $x$.

\begin{prp}\label{21}
If $X\in L(\A)$ then $(\tilde S_\A)_{I_X}=(\tilde S_{\A_X})_{I_X}=\tilde S_{\A_X/X}\otimes_\CC\CC(X)$.
In particular, $\tilde S_{\A_X/X}=(\tilde S_\A)_{I_X}^{X}$ where $X$ is considered as a translation group. 
\end{prp}

\begin{proof}
Fix $X\in L(\A)$ and let $Y$ be its orthogonal complement.
By $\Delta_X\in\CC[Y]$ we denote the defining equation of $\A_X$.
Then, by the product rule,
\[
(J_\A)_{I_X}=J_\Delta(S_{I_X}/S_{I_X}\Delta)=J_{\Delta_X}(S_{I_X}/S_{I_X}\Delta_X)=(J_{\A_X})_{I_X}.
\]
Localizing a presentation, such as \eqref{26}, at $I_X$, therefore shows that 
\begin{align*}
(I_\A)_{I_X}=(\Fit^1_S(J_\A))_{I_X}&=\Fit^1_{S_{I_X}}((J_\A)_{I_X})\\
&=\Fit^1_{S_{I_X}}((J_{\A_X})_{I_X})=(\Fit^1_S(J_{\A_X}))_{I_X}=(I_{\A_X})_{I_X}.
\end{align*}
Then we have also $(\tilde I_\A)_{I_X}=(\tilde I_{\A_X})_{I_X}$ and finally,
\begin{align*}
(\tilde S_\A)_{I_X}=(\End_{S_\A}(I_\A))_{I_X}&=\End_{S_{I_X}}((I_\A)_{I_X})\\
&=\End_{S_{I_X}}((I_{\A_X})_{I_X})=(\End_{S_{\A_X}}(I_{\A_X}))_{I_X}=(\tilde S_{\A_X})_{I_X}.
\end{align*}
This proves the first equality; the second follows since $S_{I_X}=\CC[Y]\otimes_\CC\CC(X)$.
\end{proof}

\begin{cor}
The assignment $\A\mapsto\tilde S_\A$ is a local functor.\qed
\end{cor}

We now turn our attention to $\tilde D$. 
The following result holds for any free divisor, and our proof is not specific to our situation.

\begin{prp}\label{72}
The ideals $I_\A$ and $I_D$ are stable under $\Der(-\log\A)$ and $\Der(-\log D)$ respectively.
In particular, the latter act naturally on $\tilde S_\A$ and $\tilde R_D$ respectively.
\end{prp}

\begin{proof}
Let $\omega_1,\dots,\omega_\ell\in\Omega^1(\log D)$ be the dual basis of \eqref{32}.
From
\[
R\ni d\omega_j(\delta_k,\delta_r)=d\omega_j\left(\delta_k,\sum_{i=1}^\ell K^i_r\p_{p_i}\right)=\sum_{i=1}^\ell K^i_rd\omega_j(\delta_k,\p_{p_i}),
\]
\eqref{19} and Cramer's rule, we conclude that 
\begin{align*}
I_D\ni d\omega_j(\delta_k,\Delta^2\p_{p_i})
&=\delta_k\ideal{\Delta^2\omega_j,\p_{p_i}}-\Delta^2\p_{p_i}\ideal{\omega_j,\delta_k}-\ideal{\omega_j,[\delta_k,\Delta^2\p_{p_i}]}\\
&=\delta_k(M^i_j)+\ideal{\Delta^2\omega_j,[\p_{p_i},\delta_k]-\frac{\delta_k(\Delta^2)}{\Delta^2}\p_{p_i}}\\
&\equiv\delta_k(M^i_j)\mod I_D.
\end{align*}
This proves the claim for $D$; the same argument works for $\A$ and any free divisor.
\end{proof}

\begin{rmk}
There is a transcendental argument which shows that for any divisor $D$, free or not, $\Der(-\log D)$ preserves the ideal $I_k(D)$ of $k\times k$ minors of the matrix of coefficients of a set of generators of $\Der(-\log D)$. 
It is simply that each of these ideals is invariant under biholomorphic automorphisms of $D$, since  they are Fitting ideals of the Jacobian ideal $J_D$. 
The integral flow of any vector field $\zeta\in\Der(-\log D)$ preserves $D$, and hence  $I_k(D)$, from which it follows that $\zeta\cdot I_k(D)\subset I_k(D)$.
\end{rmk}

We can improve on Proposition~\ref{21} in the analytic category.
Let $x\in X\in L(\A)$ and $y=p(x)\in p(X)=Y$.
By \cite[\S2]{Orl89}, $Y\in L(D)$ and $p\colon X\to Y$ is a covering.
By finiteness of $W$, there is a (Euclidean) $W_X$-stable neighborhood of $x$, in which the $W$-orbits are exactly the $W_X$-orbits. 
Note that $W_X$ commutes with the translation group $X$.
This gives
\[
p_x=p_{W_X,x}\times p|_X\colon V_x=(V/X)_x\times X_x\to ((V/X)/W_X)_y\times Y_y.
\]
Since our definition of $\tilde R_D$ in \eqref{71} and \eqref{18} is compatible with passing to the analytic category, we obtain the following analytic localization statement.

\begin{prp}\label{24}
Let $x\in X\in L(\A)$ and $y=p(x)\in p(X)=Y\in L(D)$, and denote by $D_Y$ the discriminant of $W_X$ on $V/X$.
Then there is an isomorphism of analytic germs $\tilde D_y\cong\tilde D_{Y,y}\times Y_y$.\qed
\end{prp}

\begin{rmk}
Saito~\cite[(3.6)]{Sai80} showed that one can always analytically trivialize the logarithmic stratification along logarithmic strata as we do in Proposition~\ref{24}. 
\end{rmk}

\begin{cor}
$\tilde\A$ is (algebraically) and $\tilde D$ (analytically) constant over logarithmic strata.\qed
\end{cor}

By \cite[\S1.8]{Hum90}, $W$ acts simply transitively on the (simple) root systems and on the Weyl chambers. 
Choosing a simple root system defining a Weyl chamber of which $X=X(x)$ is a face, shows that the Dynkin diagram of any isotropy group $W_x=W_X$ is obtained by dropping from the Dynkin diagram of $W$ the roots which are not orthogonal to $X$. 
By \cite[Prop.~2.2]{Hum90}, the connected components of the resulting  Dynkin diagram are in bijection with the irreducible factors of $W_x$.
This discussion combined with Propositions~\ref{21} and \ref{24} proves

\begin{thm}\label{61}\pushQED{\qed}
Let $X\in L(\A)$ and let $Y=p(X)$.
Let $W_1,\dots,W_r$ be the irreducible Coxeter groups whose Dynkin diagrams are the connected components of the sub-diagram of the Dynkin diagram of $W$ formed by the vertices corresponding to simple roots orthogonal to $X$.
Let $\A_1,\dots,\A_r$ and $D_1,\dots,D_r$ be their reflection arrangements and discriminants, and let $\ell_i$ be the dimension of the standard representation of $W_i$. 
Then the algebraic localization of $\A$ along $X$, and the analytic localization of $\tilde D$ along $Y$, are isomorphic, respectively, to 
the disjoint unions 
\[
\bigsqcup _{i=1}^r\tilde\A_i\times\CC^{\ell-\ell_i}\quad\quad\text{and}\quad\quad \bigsqcup_{i=1}^r\tilde D_i\times \CC^{\ell-\ell_i}.\qedhere
\]
\end{thm}

\subsection{Relation with the normalization}\label{49}

We denote the normalizations of $\A$ and $D$ by $\bar\A$ and $\bar D$ respectively.

\begin{prp}\label{43}
We have $S_\A\subseteq\tilde S_\A\subseteq\bar S_\A$ and $R_D\subseteq\tilde R_D\subseteq\bar R_D$.
\end{prp}

\begin{proof}
This follows from the finiteness and birationality of $\tilde S_\A$ and $\tilde R_D$ over $S_\A$ and $R_D$, see \eqref{15}, \eqref{13}, \eqref{42}, \eqref{25}.
\end{proof}

In the following, we describe the cases of equality in Proposition~\ref{43}. 

We begin with the case $\ell=2$ of plane curves for irreducible $W$.
By \eqref{31} and for degree reasons, this case reduces to
\begin{gather}
\label{111}K=
\begin{pmatrix}
2p_1 & hp_2 \\
hp_2 & Q
\end{pmatrix},\quad Q=ap_1^r+bp_1^sp_2,\quad r=h-1,\quad \frac h2-1=s,\\
\label{50}\Delta^2=|K|=2p_1Q-h^2p_2^2=2ap_1^h+2bp_1^{h/2}p_2-h^2p_2^2.
\end{gather}
In particular, $b=0$ if $h$ is odd.
Note that there are no further restrictions imposed on $a$ and $b$  by the requirement
\begin{equation}\label{54}
\delta_2(\Delta^2)\in R\Delta^2
\end{equation}
for $\delta_2$ from \eqref{3}.
Indeed, $\ideal{\delta_1,\delta_2}_R$ is a Lie algebra, since $[\delta _1,\delta_2]=(h-2)\delta_2$ by homogeneity.
For generic $(a,b)$, $\Delta^2$ in \eqref{50} is reduced, and hence \eqref{54} holds true by \cite[Lem.~1.9]{Sai80}.
By continuity, it holds then also for special values of $(a,b)$.

\begin{prp}\label{56}
For $\ell=2$, irreducible $W$, and odd $h\ge5$, $\tilde D\neq\bar D$.
\end{prp}

\begin{proof}\pushQED{\qed}
In this case,
\begin{equation}\label{62}
K=
\begin{pmatrix}
2p_1 & hp_2 \\
hp_2 & ap_1^r
\end{pmatrix}
\end{equation}
and \eqref{50} specializes to 
\[
\Delta^2=|K|=2ap_1^{r+1}-h^2p_2^2\equiv p_1^h-p_2^2.
\]
The normalization of $D$ is given by $p_1=t^2$ and $p_2=t^h$, and hence $g_1=\frac{p_2}{p_1}=t^{h-2}$ by \eqref{25} and \eqref{62}.
Then \eqref{42} becomes
\[
\tilde R_D=R_D[g_1]=\CC[t^2,t^{h-2}]\subsetneq \CC[t]=\bar R_D.\qedhere
\]
\end{proof}

Using Theorem~\ref{61} and Lemma~\ref{28} we find

\begin{cor}\label{59}
If $W$ contains any irreducible summand of type $H_3$, $H_4$, or $I_2(k)$ for odd $k$, then $\tilde D\neq\bar D$.
\end{cor}

\begin{proof}
For $W$ of type $I_2(k)$, we have $h=k$ and the claim follows from Proposition~\ref{56}.
For the $H_k$-types, the statement follows from Theorem~\ref{61} and the adjacency chain $H_4\to H_3\to I_2(5)$. 
\end{proof}

We write $\CC_0=S/\mm$ where $\mm$ is the graded maximal ideal in $S$.
Then $\tilde\A_0=\Spec(\tilde S_\A\otimes_S\CC_0)$ is the fiber of $\tilde\A$ over $0\in V$.

\begin{lem}\label{23}
The group $W$ acts trivially on the fiber $\tilde\A_0$ of $\tilde\A$ over $0\in V$, which contains exactly as many geometric points as the number of irreducible summands of $W$.
\end{lem}

\begin{proof}
By \eqref{15}, $\tilde S_\A\otimes_S\CC_0\cong\CC[h_1,\dots,h_{\ell-1}]$ and by Proposition~\ref{27} the $h_i$ are $W$-invariants.
This implies the first claim.
For the second statement, we may assume that $W$ is irreducible by Lemma~\ref{28}.
Then \eqref{30}, \eqref{14}, and \eqref{13} imply that $h_i$ has $w$-degree $w_\ell-w_i$.
So $\CC[h_1,\dots,h_{\ell-1}]$ is positively graded and hence $\tilde\A$ is a cone.
As it is also finite over $0\in V$ due to \eqref{15}, it must be a single geometric point as claimed.
\end{proof}

We write $\CC_x=S/\mm_x$ and $\CC_y=R/\mm_y$ where $\mm_x$ and $\mm_y$ are the maximal ideals of $S$ at $x$ and of $R$ at $y$.
Then $\tilde\A_x=\Spec(\tilde S_\A\otimes_S\CC_x)$ and $\tilde D_y=\Spec(\tilde R_D\otimes_R\CC_y)$ are the fibers of $\tilde\A$ over $x$ and of $\tilde D$ over $y$ respectively. 
Combining Propositions~\ref{21} and \ref{24}, \eqref{15}, Proposition~\ref{27}, and Lemma~\ref{23}, we find

\begin{prp}\label{44}
The fibers $\tilde\A_x$ and $\tilde D_y$, $y=p(x)$, coincide, that is,
\[
\tilde S_\A\otimes_S\CC_x=\tilde R_D\otimes_R\CC_y.
\]
They are trivial $W_x$-modules containing exactly as many geometric points as the number of irreducible summands of $W_x$.
\end{prp}

We can now refine Proposition~\ref{43} for $\A$.

\begin{cor}\label{45}\
\begin{enumerate}
\item\label{45a} $\A=\tilde\A$ exactly if $\A$ contains only one plane (or $W$ has type $A_1$). 
\item\label{45b} $\tilde\A=\bar\A$ exactly if $\A$ is Boolean (or $W$ has type $A_1\times\cdots\times A_1$). 
\end{enumerate}
\end{cor}

\begin{proof}\
\begin{asparaenum}
\item If $\#\A>1$, pick $x$ with $X(x)=X\in L_2(\A)\ne\emptyset$. 
Then $W_X$ is of type $A_1\times A_1$.
So by Proposition~\ref{44}, $\tilde\A$ has two points over $x$.
The converse is Example~\ref{46} for $\ell=1$.
\item Again one implication is Example~\ref{46}.
If $\A$ is not Boolean, then $W$ has an non-$A_1$ type irreducible summand.
By Lemma~\ref{23}, its reflection hyperplanes do not separate in $\tilde\A$.
\end{asparaenum}
\end{proof}

The analogue of Corollary~\ref{45} for $D$ is less trivial.

\begin{thm}\label{60}
$\tilde D=\bar D$ exactly if all irreducible summands of $W$ are of $ADE$-type. 
In this case, $\tilde D$ is smooth.
\end{thm}

\begin{proof}
If $W$ is of type $ADE$, then by \cite{Bri71,Slo80} $V/W$ can be identified with the base space of a versal deformation of a singularity of the same type.  
Then by \eqref{113} $\tilde D=\Sigma^0$ is a smooth space and hence $\tilde D=\bar D$.
If $W$ is reducible, with all irreducible summands of type $ADE$, then by Proposition~\ref{28} $\tilde D$ is the disjoint union of the spaces corresponding to the summands.

Conversely, consider an irreducible $W$ not of type $ADE$ and not covered by Corollary~\ref{59}, that is, of type $B_\ell$, $C_\ell$, $F_4$, or $I_2(k)$ with $k$ even.
Then there are at least two $W$-orbits in $\A$, $D$ is reducible, and $\bar D$ has at least two connected components. 
On the other hand $\tilde D$ is connected, by Lemma~\ref{23} and Proposition
\ref{44}. 
Thus $\tilde D\neq \bar D$. 
By Proposition~\ref{28} this conclusion applies to reducible $W$ also. 
\end{proof} 

\subsection{Example~\ref{108} revisited}\label{109}

In Example~\ref{108} we asserted that in the case of $A_\ell$, the space $\tilde\A$ is isomorphic to the union $L_\ell$ of the coordinate $(\ell-1)$-planes $L_{i,j}=\{x_i=x_j=0\}$ in $\CC^{\ell+1}$.
We now prove this. 
 
Recall that a space $X$ is {\em weakly normal} if every continuous function $X\to\CC$ which is holomorphic on the smooth part of $X$ is in fact holomorphic on all of $X$.

\begin{lem}
The space $ L_\ell$ is Cohen--Macaulay and weakly normal. 
\end{lem}

\begin{proof}
Cohen--Macaulayness is well known, and follows from the Hilbert--Burch theorem: the ideal $I_\ell$ of functions vanishing on $L_\ell$ is $\ideal{x_2\cdots x_{\ell+1}, x_1x_3\cdots x_{\ell+1}, \dots, x_1\cdots x_\ell}$, and it is easy to obtain this as the ideal of maximal minors of an $\ell\times(\ell+1)$ matrix. 
For weak normality, we use induction on $\ell$: the space $L_2$ is the union of the coordinate axes in $3$-space, and weak normality can easily be checked here. 
Now suppose $\ell\geq 3$ and that the statement is true for $L_{\ell-1}$, and let $f:L_\ell\to\CC$ be continuous and holomorphic on the smooth part of $L_\ell$. 
Let $x\in L_\ell$. 
If $x_j\neq 0$ then up to permutation of coordinates, the germ $(L_\ell,x)$ is equal to the product $(\CC,x_j)\times (L_{\ell-1},(x_1,\ldots,\hat x_j,\ldots, x_{\ell+1}))$. 
It follows from the induction hypothesis that $L_\ell$ is weakly normal at $x$, and therefore $f$ is holomorphic at $x$. 
Since $L_\ell$ is Cohen--Macaulay, Hartogs's Theorem holds and therefore $f$ is holomorphic also at 0.
\end{proof}

\begin{prp}
In the case of the reflection arrangement for $A_\ell$, the space $\tilde\A$ is isomorphic to $L_\ell$.
\end{prp}

\begin{proof}
We consider the standard representation space $V:=\{(x_1,\dots,x_{\ell+1})\in\CC^{\ell+1}\mid\sum_{i=1}^{\ell+1}x_i=0\}$.
The arrangement $\A$ consists of hyperplanes $H_{i,j}:=\{x\in V\mid x_i=x_j\}$. 
Let us denote by $s\colon L_\ell\to\A$ and $t\colon \tilde\A\to \A$ the natural projections. 
Recall that $s(x)=x-x^\sharp$, where $x^\sharp$ is the Hermitian orthogonal projection of $x$ to $\CC\cdot(1,\dots,1)$. 
First we establish a natural bijection $\tilde\A\to L_\ell$. 
Let $a=(a_1,\dots,a_{\ell+1})\in\A$.
Then 
\[
s^{-1}(a)=\{x\in L_\ell\mid x=a+\lambda\cdot(1,\dots,1)\text{ for some }\lambda\in\CC\}.
\]
Define an equivalence relation $\sim$ on the set $\{1,\ldots,\ell+1\}$ by $i\sim j$ if $a_i=a_j$.  Since $a+\lambda(1,\dots,1)\in L_{i,j}$ if and only if $\lambda=-a_i=-a_j$, $s^{-1}(a)$ is in bijection with the set 
\[
C:=\{\sigma\in\{1,\ldots,\ell+1\}/_\sim\mid|\sigma|\geq 2\}
\]
of non-singleton equivalence classes for $\sim$. 
For $\sigma\in C$, set
\[
V_\sigma:=\{x\in V\mid x_i=0\text{ if } i\notin\sigma\},\quad
\A_\sigma:=\{H_{i,j}\mid i, j\in\sigma, i\neq j\}.
\]
Then the $V_\sigma$, $\sigma\in C$,  are the irreducible factors of $W_a$ with corresponding reflection arrangements $\A_\sigma$.
By Proposition~\ref{21}, locally at $a$, $\tilde\A$ is isomorphic to the disjoint union of $\tilde\A_\sigma\times X(a)$, in the notation of Definition~\ref{183}.
In particular, also the fiber $t^{-1}(a)$ can be identified with $C$. 

The natural bijections $s^{-1}(a)\to t^{-1}(a)$ give rise to a natural bijection $L_\ell\to\tilde \A$ over $\A$.
We have to show now that this bijection is biholomorphic. 
Both spaces are Cohen-Macaulay, so it is enough to prove this outside a set of codimension $2$. 
It is clearly biholomorphic over smooth points of $\A$, since here the projections $\tilde\A\to\A$ and $L_\ell\to\A$ are both biholomorphic.
The codimension-$1$ singularities of $\A$ are of type $A_1+A_1$ (a normal crossing of $2$ branches, with reducible representation) and $A_2$. 
Over points of the first kind, both $\tilde\A$ and $L_\ell$ are smooth, by Corollary \ref{45}, and so the bijection is indeed biholomorphic. 
Over points of the second kind, the argument of Example~\ref{108} shows that here too the bijection is biholomorphic.
\end{proof} 

It would be interesting to know if the space $\tilde\A$ is weakly normal for other Coxeter arrangements.

\section{Dual and Hessian rank conditions}\label{7}

Let $F=S\cdot\mm_R$ be the ideal of all positive-degree $W$-invariants.
We can identify $S/F$ with a direct summand $T$ of the $W$-module $S$, and setting $S^\alpha=T\cdot p^\alpha$, we have
\begin{equation}\label{164}
S=\bigoplus_{\alpha\in\NN^\ell}S^\alpha\supset\bigoplus_{0\ne\alpha\in\NN^\ell}S^\alpha=F
\end{equation}
as a direct sum of $W$-modules, where $p=p_1,\dots,p_\ell$.
Chevalley~\cite{Che55} showed that $T$ is the regular $W$-representation (see also \cite[p.~278]{Sol64}).
Consider the $W$-modules of exterior powers
\[
E_p=\bigwedge^pV^*.
\]
Solomon~\cite[Thm.~2 and footnote~($^2$)]{Sol64} showed that the isotypic components of $S/F$ of type $E_1\cong V^*$ and $E_{\ell-1}\cong V\otimes \det V$ are the direct sums of the projections to $S/F$ of the $W$-modules
\begin{align}\label{165}
J^j&=\ideal{\p_{x_k}(p_j)\mid k=1,\dots,\ell}_\CC,\\
\nonumber M^j&=\ideal{m^j_k\mid k=1,\dots,\ell}_\CC,\quad j=1,\dots,\ell,
\end{align}
respectively.
We may and will assume that $J^j\subset T$ and $M^j\subset T$.
By \eqref{84} and \eqref{215}, $D_j$ is the homogeneous degree of $M^j$, while $m_j$ is the homogeneous degree of $J^j$.

Let us recall the construction from the proof of \cite[Thm.~2]{Sol64}:
We denote by $I(-)$ the $W$-invariant part.
By \cite{Sol63}, the space of $W$-invariant differential forms on $V$ is 
\[
I(S\otimes E_p)=\sum_{i_1<\cdots<i_p}R\cdot dp_{i_1}\wedge\dots\wedge dp_{i_p}.
\]
Solomon~\cite[p.~282]{Sol64} considers the case where $W$ is the Weyl group of a Lie group acting on $V$; then the Killing form induces a self-duality $E_p\cong E_p^*$.
We are only interested in the cases $p=1$ and $p=\ell-1$, where both irreducibility and self-duality of $E_p$ are trivial\footnote{$E_1\cong V^*$ is self-dual due to the $W$-invariant form $p_2$ on $V$, and hence irreducible, since $V$ is irreducible. Because $\det(V)^{\otimes2}\cong\CC$ is the trivial representation, $E_{\ell-1}\cong E_1^*\otimes E_\ell\cong V\otimes \det(V)$ is self-dual.
For the same reason and irreducibility of $V$, $I(V\otimes \det(V)\otimes(V\otimes \det(V))^*)=I(V\otimes V^*)=1$, and hence $E_{\ell-1}$ is irreducible.}.
The self-duality of $E_p$ induces a $W$-isomorphism $S/F\otimes E_p\cong\Hom_\CC(E_p,S/F)$ and hence an isomorphism 
\begin{equation}\label{184}
I(S/F\otimes E_p)\cong\Hom_W(E_p,S/F).
\end{equation}
The image of $dp_i$ in $\Hom_W(E_p,S/F)$ has image $J^i$, and the image of $dp_1\wedge\dots\wedge\widehat{dp_i}\wedge\dots\wedge dp_\ell$ has image $M^i$.

Using \eqref{164}, 
\begin{equation}\label{171}
\bigoplus_{j=1}^\ell\bigoplus_{\alpha\in\NN^\ell}M^jp^\alpha\quad\text{and}\quad\bigoplus_{j=1}^\ell\bigoplus_{0\ne\alpha\in\NN^\ell}M^jp^\alpha
\end{equation}
are the isotypic components of type $E_{\ell-1}$ of $S$ and $F$ respectively.
In particular, we have the following

\begin{lem}\label{172}
The isotypic component of $F$ of type $E_{\ell-1}$ lies in $ F\cdot I_\A$.\qed
\end{lem}

It follows that (grc) can be checked modulo $F$. 

\begin{dfn}
We say that the {\em graded rank condition mod $F$} holds for $\A$ if $M^j\subset S\cdot M^\ell +F$ for all $j=1,\dots,\ell-1$. 
\end{dfn}

\begin{lem}\label{173}
The graded rank condition mod $F$ is equivalent to the graded rank condition for $\A$.
\end{lem}

\begin{proof}
Consider the maps of $W$-modules
\begin{equation}\label{175}
\xymat{
\phi_*\colon\Hom_\CC(M^j,M^\ell\otimes_\CC S_{D_j-D_\ell})\ar[r]^-{\mu_*}&\Hom_\CC(M^j,S_{D_j})\ar[r]^-{\pi_*}&\Hom_\CC(M^j,T_{D_j})
}
\end{equation}
induced by the composition of $W$-linear maps $\phi=\pi\circ\mu$, where
\[
\mu\colon S\otimes_\CC S\to S\quad\text{and}\quad\pi\colon S\onto S/F=T
\]
are the product in $S$ and the canonical projection to $T$.
By hypothesis, there is a $\CC$-linear map $\alpha\in\Hom_\CC(M^j,M^\ell\otimes_\CC S_{D_j-D_\ell})$ such that $\phi_*(\alpha)\in\Hom_\CC(M^j,M^j)$ is the identity map.
Now averaging yields 
\[
\gamma=\alpha^\#\in\Hom_W(M^j,M^\ell\otimes_\CC S_{D_j-D_\ell}),\quad\phi_*(\gamma)=\id_{M^j}.
\]
Using Lemma~\ref{172}, we find that
\[
\mu_*(\gamma)-\id_{M^j}\in\Hom_W(M^j,F)=\Hom_W(M^j,F\cdot I_\A).
\]
This proves that
\[
I_\A\subset S\cdot M^\ell+F\cdot I_\A,
\]
and hence $I_\A=S\cdot M^\ell$ by Nakayama's lemma.
\end{proof}

By Solomon's result mentioned above, the $W$-equivariant Gorenstein pairing on $S/F$ induces a non-degenerate pairing of the isotypic components of type $E_1$ and $E_{\ell-1}$ into the unique irreducible summand of type $E_\ell\cong\det(V)$,
\[
\bigoplus_{i=1}^\ell J^i\otimes\bigoplus_{j=1}^\ell M^j\to\CC\cdot\Delta.
\]
Since the element 
\[
\sum_{i=1}^\ell\p_{x_i}(p_j)\otimes m^j_i\in J^j\otimes M^j
\]
maps to $\Delta=\det J$ by Laplace expansion of the determinant along the $j$'th row, we obtain induced non-degenerate pairings
\begin{equation}\label{177}
J^j\otimes M^j\to\CC\cdot\Delta,\quad j=1,\dots,\ell.
\end{equation}
For $j<k$, we have
\begin{align}\label{180}
\Hom_W(J^j,J^k)\cong\End_W(E_1)&\cong\End_W(E_{\ell-1}^*\otimes E_\ell)\\
\nonumber&\cong\End_W(E_{\ell-1}^*)\cong\Hom_W(M^k,M^j),
\end{align}
where $\mu_*(\alpha)\in\Hom_W(J^j,J^k)$ induced by $\alpha\in\Hom_W(J^j,J^j\otimes S_{m_k-m_j})$ corresponds to $\mu_*(\beta)\in\Hom_W(M^k,M^j)$ induced by $\beta=\alpha^t\in\Hom_W(M^k,M^k\otimes S_{D_j-D_k})$.
Note here that $m_k-m_j=D_j-D_k$ by \eqref{215}.
Because of the non-degenerate $W$-pairing \eqref{177}, $\mu_*(\alpha)$ is an isomorphism exactly if $\mu_*(\beta)$ is an isomorphism.

\begin{dfn}\label{182}
We say that the \emph{dual (graded) rank condition} (drc) holds for $\A$ if $J^\ell\subset S\cdot J^j+F$ for all $j=1,\dots,\ell-1$.
\end{dfn}

\begin{rmk}
The definition of (drc) is given as an equality in $S/F$ because in general $J^\ell\not\subset S\cdot J^j$, though the inclusion holds trivially for $j=1$. 
\end{rmk}

\begin{lem}\label{179}
The graded rank condition mod $F$ is equivalent to the dual rank condition for $\A$.
\end{lem}

\begin{proof} 
We show that (grc) mod $F$ implies (drc). The opposite implication is proved in just the same way.
Fix $j\in\{1,\dots,\ell-1\}$.  
By (grc) mod $F$, there is a $\beta\in\Hom_\CC(M^j,M^\ell\otimes S_{D_j-D_\ell})$ inducing the identity map $\id_{M^j}=\pi_*\mu_*(\beta)\in\Hom_\CC(M^j,M^j)$.
By averaging, we can turn $\beta$ into a $W$-homomorphism. 
The homomorphism $\mu_*(\beta)$ is non-zero modulo $F$ and \eqref{180} yields a corresponding dual map $\mu_*(\alpha)\in\Hom_W(J^j,S_{D_\ell})$ induced by $\alpha:=\beta^t\in\Hom_W(J^\ell,J^j\otimes S_{m_\ell-m_j})$. 
This shows that (drc) holds. 
\end{proof}

By Lemma~\ref{173}, we deduce the following equivalence that combined with Theorems~\ref{206} and \ref{160} and Lemma~\ref{161} proves Corollary~\ref{110}.

\begin{prp}\label{119}
The dual graded rank condition is equivalent to the (graded) rank condition for $\A$.\qed
\end{prp}

The following property refines (grc) by a statement about the $S$-coefficients of $J^j$ in the condition in Definition~\ref{182}.
By \cite[(2.14)~Lem.]{OS88}, the Hessian
\[
\Hess(p)\colon\Der_S\to\Omega_S^1,\quad \Hess(p)(\delta):=\sum_{i=1}^\ell\delta(\p_{x_i}(p))dx_i,
\]
is $W$-equivariant for $p\in R$.
Note that $\Hess(p_1)$ is a $W$-isomorphism which induces our identification of $dp_i$ with a vector field $\eta_i$ in \eqref{40}.
By abuse of notation, we identify
\[
\Hess(p)=\Hess(p)\circ\Hess(p_1)^{-1}\in\End_W(\Omega_S^1)
\]
for $p\in R$.
Using $\Omega_S^1=S\otimes E_1$ and passing to the quotient by $F$, $\Hess(p)$ then induces an element of $\End_W(S/F\otimes E_1)$ and hence of $\End_W(I(S/F\otimes E_1))$.
By \eqref{184}, $\Hess(p)$ thus induces a map
\[
\hbar(p)\in\End_W(\Hom_W(E_1,S/F))
\]
which operates on $W$-submodules of type $V^*$ by passing to the image in $\Hom_W(E_1,S/F)$.

\begin{dfn}\label{185}
We say that the \emph{Hessian (dual graded) ring condition} (Hrc) holds for $\A$ if, for any $j$, there is an $i$, such that $m_i+m_j=w_\ell$ and $\Hess(p_i)(\eta_j)\not\in F\Omega^1_S$.
In case $m_1,\dots,m_\ell$ are pairwise different, this means that $\Hess(p_i)(\eta_{\ell-i+1})\not\in F\Omega^1_S$.
\end{dfn}

\begin{lem}\label{181}
The Hessian rank condition implies the dual ring condition for $\A$.
\end{lem}

\begin{proof}
(Hrc) means that $\hbar(p_i)(J^j)\subset(S/F)_{m_\ell}$ is non-zero.
By $W$-equivariance of $\hbar(p_i)$, the latter is then a non-trivial $W$-submodule of $(S/F)_{m_\ell}$ of type $E_1$.
Then it must coincide with $J_\ell$, which is the only such $W$-module in this degree by \eqref{30}.
\end{proof}

\begin{thm}\label{221}
The Hessian rank condition holds for $\A$ if $W$ is not of type $E_6$, $E_7$, or $E_8$.
\end{thm}

\begin{proof}
It is clear that $\Hess(p_i)(\eta_1)=dp_i$, so (Hrc) holds trivially in dimension $\ell=2$.
For the $A$- and $B$-types, it is an easy exercise to verify (Hrc) using \cite[\S3.12]{Hum90}.
In case of $F_4$, $H_3$ and $H_4$, {\tt Macaulay2}~\cite{M2} calculations, based on the formul\ae\ for basic invariants given by Mehta~\cite{Meh88}, show that (Hrc) holds for $\A$.

Let us now prove (Hrc) for $W$ of type $D_\ell$.
By \cite[\S3.12]{Hum90}, the basic invariants can be chosen as the power sums
\[
p_k=\frac1{2k}(x_1^{2k}+\dots + x_{\ell}^{2k}),\quad k=1,\dots,\ell-1,
\]
together with $p_\ell=x_1\cdots x_\ell$.
Note the change of notation turning $p_{\ell-1}$ into the highest degree invariant.
It is easy to check that $D(p_i)\circ\Hess(p_{\ell-i})\equiv D(p_{\ell-1})\mod\CC^*$ for $i=1,\dots,\ell-2$.
We now replace $p_{\ell-1}$ by the invariant polynomial 
\[
\hat p_{\ell-1}(x_1,\dots,x_\ell)=D(p_\ell)\cdot D(p_\ell)=\sum_{j=1}^\ell x_1^2\cdots\widehat{x_j^2}\cdots x_\ell^2\in R
\]
of the same degree.
We claim that $p_{\ell-1}\equiv\hat p_{\ell-1}\mod F^2+\CC^*$. 
In the evident equality 
\[
2\cdot D(p_\ell)\circ\Hess(p_\ell)=D(\hat p_{\ell-1})
\]
we can then replace $\hat p_{\ell-1}$ by $p_{\ell-1}$ modulo $F$, completing the proof of (Hrc). 

In order to verify the claim, let $\rho$ be a primitive $2(\ell-1)$'th root of unity and set $a=(\rho,\rho^2,\dots,\rho^{\ell-1},0)$.
Then all of our basic invariants except for $p_{\ell-1}$ vanish at $a$, while $\hat p_{\ell-1}(a)\neq 0\neq p_{\ell-1}(a)$. 
Since $\deg\hat p_{\ell-1}=\deg p_{\ell-1}>\deg p_i$ for all $i\ne\ell-1$ by \eqref{30}, the claim follows.
\end{proof}

Computing limitations oblige us to leave open the following conjecture.

\begin{cnj}\label{118}
The Hessian rank condition holds for $\A$ if $W$ is of type $E_6$, $E_7$, or $E_8$.
\end{cnj}

\begin{add}\label{120}
After publication of this article in \cite{GMS12}, Hiroaki Terao kindly informed us that Conjecture~\ref{118} follows from results of \cite{ST98}.
He outlined the following direct proof of Conjecture~\ref{118}:
By Theorem~\ref{9} and a degree argument using \eqref{84}, \eqref{30}, and \eqref{83}, we have
\[
K=JJ^t=(D(p_i)\cdot D(p_j))\equiv
\begin{pmatrix}
0&\cdots&c_\ell p_\ell\\
\vdots&\rdots&\vdots\\
c_1p_\ell&\cdots&0
\end{pmatrix}
\mod\ideal{p_1,\dots,p_{\ell-1}}
\]
with $c_1\cdots c_\ell\ne 0$ where, as before, $D(p)$ means the gradient of $p$.
Write 
\[
D(p_i)\circ\Hess(p_j)=\sum_kr_{i,j}^kD(p_k)
\]
with $r_{i,j}^k\in\CC$.
Then right-multiplying by $(x_1,\dots,x_\ell)^t$ and using the Euler identity, yields
\[
m_jc_jp_\ell=m_jD(p_i)\cdot D(p_j)=\sum_km_kr_{i,j}^kp_k.
\]
Specializing to $j:=\ell-i+1$ such that $m_i+m_j=h$ by \eqref{83}, this implies $r_{i,j}^\ell\ne0$ and hence (Hrc).\qed
\end{add}

\section{Free and adjoint divisors}\label{96}

In \cite{MS10} a new class of free divisors was constructed using the recipe ``discriminant + adjoint''.  
If $D$ is the discriminant in the base of a miniversal deformation of a weighted homogeneous hypersurface singularity (subject to some numerical conditions on the weights) and $D'$ is an adjoint divisor, in the sense that the pull-back of $D'$ to the normalization $\Sigma^0$ of $D$ is the conductor of the ring extension $\OO_D\to \OO_{\Sigma^0}$, then $D+D'$ is a free divisor (\cite[Thm.~1.3]{MS10}). 
The singularities to which this applies include those of type $ADE$.
In this section we point out that essentially the same construction works for the other Coxeter groups. 
We have to replace the normalization $\bar D$ by the space $\tilde D$ of Lemma~\ref{194} (though recall that $\bar D=\tilde D$ for Coxeter groups of type $ADE$), and take, as $D'$, a divisor pulling back to the conductor of the ring extension $\OO_D\into\OO_{\tilde D}$. 
The construction lifts to the representation space $V$, giving a new free divisor strictly containing the reflection arrangement.

We keep the notations from Section~\ref{80} and work in standard coordinates as in \eqref{4}.

\begin{lem}\label{89}
With a suitable choice of basic invariants $p_1,\dots,p_\ell$, the linear part $ \bar K$ of the Saito matrix $K=JJ^t$ of $D$ from the exact sequence \eqref{26} is symmetric of the form
\begin{equation}\label{64}
\bar K=
\begin{pmatrix}
w_1 p_1 & w_2 p_2 & \cdots & \cdots & w_{\ell-1} p_{\ell-1} &w_\ell p_\ell \\
w_2 p_2 & \star & \cdots & \star &\alpha_{\ell-1} p_\ell & 0 \\
\vdots & \vdots & \rdots & \rdots & \rdots & \vdots \\
\vdots & \star & \rdots & & & \vdots \\
w_{\ell-1} p_{\ell-1} &\alpha_2 p_\ell & \rdots & & & \vdots \\
w_\ell p_\ell & 0 & \cdots & \cdots & \cdots & 0
\end{pmatrix}
\end{equation}
where $\alpha_2,\dots,\alpha_{\ell-1}\in\CC^*$ with $\alpha_i=\alpha_{\ell+1-i}$.
Moreover the only entries in this matrix equal to non-zero constant multiples of $p_\ell$ lie along the anti-diagonal. 
\end{lem}

\begin{rmk} 
This matrix shows the linearized convolution of the basic invariants $p_1,\dots, p_\ell$ as described in \cite{Arn79}. 
\end{rmk}

\begin{proof}
The first row and column of \eqref{64} can be read from \eqref{31}. 
It remains to show the triangular form of $\bar K$ and that the anti-diagonal entries, and only these, are non-zero constant multiples of $p_\ell$. 
By inspection, the degree of $K^i_j$ is $w_i+w_j-w_1$. 
By \eqref{84}, \eqref{30} and \eqref{83}, the degree of $K^i_j$ with $i+j=\ell+1$ equals $h=w_\ell$, and hence $\bar K^i_j=\alpha_jp_\ell$ for some $\alpha_j\in\CC$.
Provided $W$ is not of type $D_{2k}$, the degrees $w_1,\dots,w_\ell$ of the basic invariants are pairwise distinct. 
It follows that:
\begin{itemize}
\item All $K^i_j$ with $i+j>\ell+1$ have degree strictly between $w_\ell$ and $2w_\ell$ and hence have a linear part equal to zero.
In particular, $\bar K$ has the claimed triangular shape.
\item All $K^i_j$ with $i+j<\ell+1$ have degree less than $w_\ell$, and hence do not involve $p_\ell$.
\end{itemize}
But by \eqref{19}, \eqref{3}, and Theorem~\ref{9}, $\det K=\Delta^2$ is a monic polynomial of degree $\ell$ in $p_\ell$. 
It follows that $\alpha_2\cdots\alpha_{\ell-2}\ne0$.
Finally the symmetry property $\alpha_i=\alpha_{\ell+1-i}$ comes from the symmetry of $K$.

In the case of $D_{2k}$, the same argument shows that the $p_\ell$-coefficient matrix of $\bar K$ is a constant symmetric anti-diagonal block matrix, where $i$ and $j$ are in the same block exactly if $w_i=w_j$.
By the procedure in the proof of \cite[Lem.~3.6]{MS10} it can be turned into a symmetric anti-diagonal matrix by linear algebra on the basic invariants.
\end{proof}

\begin{rmk}\label{88}
By \eqref{30}, the minor $M^\ell_\ell$ is not changed by the change of basic invariants in Proposition~\ref{89}.
\end{rmk}

For $\bar K$ as in \eqref{64}, we set
\[
(\bar M^i_j):=\ad(\bar K),\quad \bar I_D:=\ideal{\bar M_1^\ell,\dots,\bar M_\ell^\ell}.
\]
Note that because (rc) holds, $\bar I_D=\ideal{\bar M^i_j\mid 1\leq i,j\leq\ell}$.

\begin{lem}\label{66}
$dM^\ell_\ell(\Der(-\log D))=I_D$.
\end{lem}

\begin{proof}
The strategy is the same as in the proof of the analogous result in \cite{MS10}.
We replace $\delta_i$ by its linear part $\bar\delta_i$ whose coefficients are in the $i$'th row/column of $\bar K$ in \eqref{64}.
Then it suffices to prove that the inclusion
\begin{equation}\label{90}
d\bar M^\ell_\ell(\ideal{\bar\delta_1,\dots,\bar\delta_\ell})\subseteq\bar I_D.
\end{equation}
obtained from Proposition~\ref{72} is an equality.
The polynomial expansion of the minor $\bar M^\ell_{\ell-i+1}$ contains the distinguished monomial $p_ip_\ell^{\ell-2}$ with non-zero coefficient. 
This monomial does not appear in the expansion of $\bar M^\ell_j$ for $j\neq i$. 
In particular the expansion of $\bar M^\ell_\ell$ contains the monomial $p_1p_\ell^{\ell-2}$, with coefficient $(-1)^{\ell-2}\iota w_1\alpha$, where $\iota$ is the sign of the order-reversing permutation of $1,\dots,\ell-1$, and $\alpha:=\alpha_2\cdots\alpha_{\ell-1}.$

We claim that $d\bar M^\ell_\ell(\bar\delta_i)$ contains the monomial $p_ip_\ell^{\ell-2}$ with non-zero coefficient, and no other of the distinguished monomials. 
This shows that \eqref{90} is an equality and proves the lemma. 

Contributions to the coefficient of $p_jp_\ell^{\ell-2}$ in the expansion of $d\bar M_\ell^\ell(\bar\delta_i)$ arise as follows:

\begin{enumerate} 

\item By applying the derivation $p_j\p_{p_1}$ to the monomial $p_1p_\ell^{\ell-2}$. 
This happens only when $i=j$, and in this case the resulting contribution to the coefficient of $p_jp_\ell^{\ell-2}$ is 
\[
\delta_{i,j}(-1)^{\ell-2}\iota w_iw_1\alpha.
\]

\item
By applying the derivation $p_\ell\p_{p_k}$ to the monomial $p_jp_kp_\ell^{\ell-3}$.
This derivation appears in $\bar\delta_i$ only if $k=\ell-i+1$, and then with coefficient $\alpha_i$; also this monomial appears in $\bar M^\ell_\ell$ only if $k=\ell-j+1$, and hence $i=j$.
If $2j=\ell+1$, the monomial $p_jp_{\ell-i+1}p_\ell^{\ell-3}$ appears in the expansion of $\bar M^\ell_\ell$ with coefficient
\[
\delta_{i,j}(-1)^{\ell-1}\iota w_jw_{\ell-j+1}\alpha/\alpha_j,
\]
otherwise, it appears twice with that coefficient.
The resulting contribution to the coefficient of $p_jp^{\ell-2}_\ell$ in $d\bar M^\ell_\ell(\bar\delta_i)$ is
\[
\delta_{i,j}(-1)^{\ell-1}\iota\alpha w_jw_{\ell-j+1}
\]
if $2j=\ell+1$, or twice this if $2j\ne\ell+1$.

\end{enumerate}

Therefore $p_jp_\ell^{\ell-2}$ can appear in $d\bar M^\ell_\ell(\delta_i)$ with non-zero coefficient only if $i=j$, and in this case the coefficient is non-zero provided 
\[
\begin{cases}
w_1\ne w_j, & \text{if }2j=\ell+1,\\
w_1\ne2w_{\ell-j+1}, & \text{if }2j\ne\ell+1.
\end{cases}
\]
These conditions hold by \eqref{30}.
\end{proof}

\begin{thm}\label{68}
Let $D'=\{M^\ell_\ell=0\}$. 
Then $D+D'$ is a free divisor. 
\end{thm}

\begin{proof}
Here the proof is identical to the proof of the comparable result of \cite[Prop.~3.10]{MS10}. 
By Lemma~\ref{66}, there are vector fields $\tilde\delta_1,\dots,\tilde\delta_\ell\in\Der(\log D)$ such that 
\begin{equation}\label{91}
dM_\ell^\ell(\tilde\delta_i)=M^\ell_i.
\end{equation} 
We may take $\tilde\delta_\ell$ equal to a constant multiple of the Euler vector field $\delta_1$. 
Since $\delta_1,\dots,\delta_\ell$ is a basis of $\Der(-\log D)$, there exist $B^i_j\in R$ such that $\tilde\delta_i=\sum_{j=1}^\ell B^j_i\delta_j$. 
By the proof of Lemma~\ref{66}, the matrix $B=(B^i_j)$ is invertible. 
Note that the Saito matrix of the basis $\tilde\delta_1,\dots,\tilde\delta_\ell$ is then $KB$.
Let $K'$ be obtained from the matrix $K$ by deleting its last column. 
The columns of $K'$ give relations among the generators $M^\ell_1,\dots,M^\ell_\ell$ of $I_D$, by Cramer's rule. 

For each relation $\sum_{i=1}^\ell\lambda_iM^\ell_i=0$, \eqref{91} gives
\[
\sum_{i=1}^\ell\lambda_i\tilde\delta_i(M^\ell_\ell)=dM^\ell_\ell\bigl(\sum_{i=1}^\ell\lambda_i\tilde\delta_i\bigr)=\sum_{i=1}^\ell\lambda_iM^\ell_i=0,
\]
so 
\[
\sum_{i=1}^\ell\lambda_i\tilde\delta_i\in\Der(-\log D)\cap\Der(-\log D')=\Der(-\log (D+D')).
\]
Because $\tilde\delta_\ell$ is a scalar multiple of $\delta_1$, we also have $\tilde\delta_\ell\in\Der(-\log (D+D'))$. 
Let $K''$ denote the matrix formed by adjoining to $K'$ the extra column $(0,\dots,0,1)^t$. 
Thus the columns of the $\ell\times\ell$ matrix $KBK''$ are the coefficients of vector fields in $\Der(-\log (D+D'))$, and $\det(KBK'')\equiv\Delta^2M^\ell_\ell\mod\CC^*$ where $\Delta^2=\det K$ is a reduced equation for $D$.
Now provided 
\begin{enumerate}
\item\label{86} $M^\ell_\ell$ is reduced, and 
\item\label{87} $M^\ell_\ell$ and $\Delta^2$ have no common factor,
\end{enumerate}
it follows from Saito's criterion that $D+D'$ is a free divisor, and the vector fields represented by the columns of $KBK''$ form a free basis for $\Der(-\log(D+D'))$. 

By \cite[Cor.~3.15]{MP89}, $M^\ell_\ell$ generates (over $\tilde R_D$) the conductor ideal of the map $\tilde D\to D$.
It follows that $D\cap D'=V(I_D)=\Sing(D)$ has codimension $2$, and hence \eqref{87} holds.
It suffices to check \eqref{86} at generic points of $\Sing(D)$.
Using Proposition~\ref{24}, this reduces to checking \eqref{86} in the case $\ell=2$ discussed in Section~\ref{49}. 
But in this case $M^2_2=2p_1$ is reduced by \eqref{111}.
\end{proof}

\begin{cor}\label{55}
$\A+p^{-1}(D')$ is a free divisor. 
\end{cor}

\begin{proof}
We continue with the notation of the proof of Theorem~\ref{68}.
Consider the vector fields represented by the columns of $J^t(BK'')\circ p$. 
Since $JJ^tBK''=KBK''$, these vector fields are lifts to $V$ of the vector fields represented by the columns of $KBK''$; they are therefore logarithmic with respect to $p^{-1}(D')$. 
Since they are linear combinations of the columns of $J^t$ they are logarithmic with respect to $\A$, and thus with respect to $\A+p^{-1}(D')$. 

By \eqref{19}, $\det J=\Delta$ is a reduced equation of $\A$.
Since $\det K''=\pm M^\ell_\ell$ is reduced and, along $V(M^\ell_\ell)$, $p$ is generically a submersion (for the critical set of $p$ is $\A$, which meets $V(M^\ell_\ell\circ p)$ only in codimension $2$), $\det(K''\circ p)$ is a reduced equation for $V(M^\ell_\ell\circ p)$.
As $\det B\in\CC^*$, $\det(J^t(BK'')\circ p)$ is therefore a reduced equation for $\A+p^{-1}(D')$, and the corollary follows by Saito's criterion.  
\end{proof}

\begin{exa}
The reflection arrangement for $A_n$ consists of the intersection of $V:=\{\sum_{i=1}^{n+1}x_i=0\}\subset\CC^{n+1}$ with the union of the hyperplanes $\{x_i=x_j\}$. 
For $A_2$, the composite equation $M^\ell_\ell\circ p$ defining $p^{-1}(D')$ in Corollary~\ref{55} is equal, on $V$, to the second elementary symmetric function, $\sigma_2$. 
For $A_3$, this becomes $8\sigma_2\sigma_4-9\sigma_3^2-2\sigma_2^3$.
\end{exa}

\begin{add}\label{121}
The proof of Theorem~\ref{68} given here is not the one which appeared in the published version \cite{MS13} of \cite{MS10}, and is greatly simplified by the argument given there in Proposition~3.9 and the proof of Theorem~1.2 which follows it.
\end{add}

\bibliographystyle{amsalpha}
\bibliography{rccox}
\end{document}